\pdfoutput=1
\documentclass{amsart} % reqno places equation numbers on the right

\newif\ifpersonal

\usepackage{amsmath,amscd,amsthm,amssymb,mathrsfs,mathtools,bm} % math related
\usepackage{stmaryrd}

\usepackage{microtype,lmodern,textcomp} % latex technical issues
\usepackage{enumitem,comment,braket,xspace,tikz-cd,svg} % utilities
\usepackage[utf8]{inputenc} % input encoding
\usepackage[T1]{fontenc} % font encoding 4518376
\usepackage[centering,vscale=0.7,hscale=0.7]{geometry}
\usepackage[hidelinks]{hyperref}
\usepackage[capitalize]{cleveref}

\allowdisplaybreaks

\numberwithin{equation}{section}
\theoremstyle{plain}
\newtheorem{theorem}[equation]{Theorem}
\newtheorem{lemma}[equation]{Lemma}
\newtheorem*{lemma*}{Lemma}

\newtheorem*{claim*}{Claim}

\newtheorem{proposition}[equation]{Proposition}
\newtheorem*{proposition*}{Proposition}
\newtheorem{corollary}[equation]{Corollary}
\newtheorem*{corollary*}{Corollary}
\theoremstyle{definition}
\newtheorem{definition}[equation]{Definition}
\newtheorem{definition-theorem}[equation]{Definition-Theorem}
\newtheorem{definition-lemma}[equation]{Definition-Lemma}
\newtheorem{construction}[equation]{Construction}
\newtheorem{assumption}[equation]{Assumption}
\newtheorem{notation}[equation]{Notation}

\newtheorem{remark}[equation]{Remark}
\newtheorem*{remark*}{Remark}

\theoremstyle{definition}

\numberwithin{equation}{section}

\renewenvironment{abstract}{%
  \quotation
  \small
  \textbf{\textit{\abstractname.}} % with a normal space
}{\endquotation}
\ifpersonal
\newcommand{\personal}[1]{\textcolor[rgb]{0,0,1}{(Personal: #1)}}

\newcommand{\todo}[1]{\textcolor{red}{(Todo: #1)}}
\else
\newcommand{\personal}[1]{\ignorespaces}
\newcommand{\discussion}[1]{\ignorespaces}
\newcommand{\todo}[1]{\ignorespaces}
\fi

\newcommand{\bbA}{\mathbb A}

\newcommand{\bbC}{\mathbb C}

\newcommand{\bbG}{\mathbb G}

\newcommand{\bbN}{\mathbb N}

\newcommand{\bbQ}{\mathbb Q}
\newcommand{\bbR}{\mathbb R}

\newcommand{\bbZ}{\mathbb Z}
\newcommand{\bbk}{\Bbbk}

\newcommand{\fD}{\mathfrak D}

\newcommand{\fM}{\mathfrak M}

\newcommand{\fU}{\mathfrak U}

\newcommand{\fX}{\mathfrak X}

\newcommand{\fd}{\mathfrak d}

\newcommand{\cA}{\mathcal A}

\newcommand{\cF}{\mathcal F}

\newcommand{\cM}{\mathcal M}

\newcommand{\cO}{\mathcal O}
\newcommand{\cP}{\mathcal P}

\newcommand{\cU}{\mathcal U}

\newcommand{\cX}{\mathcal X}
\newcommand{\cY}{\mathcal Y}
\newcommand{\cZ}{\mathcal Z}

\newcommand{\bA}{\mathbf A}

\newcommand{\bM}{\mathbf M}

\makeatletter
\let\save@mathaccent\mathaccent
\newcommand*\if@single[3]{%
	\setbox0\hbox{${\mathaccent"0362{#1}}^H$}%
	\setbox2\hbox{${\mathaccent"0362{\kern0pt#1}}^H$}%
	\ifdim\ht0=\ht2 #3\else #2\fi
}
\newcommand*\rel@kern[1]{\kern#1\dimexpr\macc@kerna}
\newcommand*\widebar[1]{\@ifnextchar^{{\wide@bar{#1}{0}}}{\wide@bar{#1}{1}}}
\newcommand*\wide@bar[2]{\if@single{#1}{\wide@bar@{#1}{#2}{1}}{\wide@bar@{#1}{#2}{2}}}
\newcommand*\wide@bar@[3]{%
	\begingroup
	\def\mathaccent##1##2{%
		\let\mathaccent\save@mathaccent
		\if#32 \let\macc@nucleus\first@char \fi
		\setbox\z@\hbox{$\macc@style{\macc@nucleus}_{}$}%
		\setbox\tw@\hbox{$\macc@style{\macc@nucleus}{}_{}$}%
		\dimen@\wd\tw@
		\advance\dimen@-\wd\z@
		\divide\dimen@ 3
		\@tempdima\wd\tw@
		\advance\@tempdima-\scriptspace
		\divide\@tempdima 10
		\advance\dimen@-\@tempdima
		\ifdim\dimen@>\z@ \dimen@0pt\fi
		\rel@kern{0.6}\kern-\dimen@
		\if#31
		\overline{\rel@kern{-0.6}\kern\dimen@\macc@nucleus\rel@kern{0.4}\kern\dimen@}%
		\advance\dimen@0.4\dimexpr\macc@kerna
		\let\final@kern#2%
		\ifdim\dimen@<\z@ \let\final@kern1\fi
		\if\final@kern1 \kern-\dimen@\fi
		\else
		\overline{\rel@kern{-0.6}\kern\dimen@#1}%
		\fi
	}%
	\macc@depth\@ne
	\let\math@bgroup\@empty \let\math@egroup\macc@set@skewchar
	\mathsurround\z@ \frozen@everymath{\mathgroup\macc@group\relax}%
	\macc@set@skewchar\relax
	\let\mathaccentV\macc@nested@a
	\if#31
	\macc@nested@a\relax111{#1}%
	\else
	\def\gobble@till@marker##1\endmarker{}%
	\futurelet\first@char\gobble@till@marker#1\endmarker
	\ifcat\noexpand\first@char A\else
	\def\first@char{}%
	\fi
	\macc@nested@a\relax111{\first@char}%
	\fi
	\endgroup
}
\makeatother

\newcommand{\hA}{\widehat A}

\newcommand{\hR}{\widehat R}
\newcommand{\hS}{\widehat S}

\newcommand{\hGamma}{\widehat\Gamma}

\newcommand{\tA}{\widetilde A}

\newcommand{\tD}{\widetilde D}

\newcommand{\tS}{\widetilde S}

\newcommand{\tX}{\widetilde X}

\newcommand{\tbeta}{\widetilde\beta}

\newcommand{\dbb}[1]{[\![#1]\!]}
\newcommand{\dbp}[1]{(\!(#1)\!)}

\newcommand\virt{\mathrm{virt}}
\newcommand\ev{\mathrm{ev}}

\newcommand\vir{\mathrm{vir}}
\newcommand{\sm}{\mathrm{sm}}
\newcommand{\trop}{\mathrm{trop}}

\DeclareMathOperator{\dom}{dom} % domain of a stable map.
\DeclareMathOperator{\NB}{NB} % bend of a piecewise linear map.
\DeclareMathOperator{\NT}{NT} % nodal metric trees.
\newcommand{\tr}{\mathrm{tr}} % transverse locus.

\newcommand{\hh}{\widehat{h}}
\newcommand{\hbeta}{\widehat{\beta}}
\newcommand{\hdelta}{\widehat{\delta}}
\newcommand{\tSigma}{\widetilde{\Sigma}}
\newcommand{\bw}{\mathbf{w}}
 
\newcommand{\ghost}[1]{\overline{#1}}
\newcommand{\fs}{\mathrm{fs}} % fine and saturated
\newcommand{\gp}{\mathrm{gp}} % group associated to a monoid.

\DeclareMathOperator{\Aut}{Aut}
\DeclareMathOperator{\coker}{coker}
\newcommand{\tors}{\mathrm{tors}} % torsion part. 

\newcommand{\bsigma}{\boldsymbol{\sigma}}
\newcommand{\bu}{\boldsymbol{\mathrm{u}}}
\newcommand{\btau}{\boldsymbol{\tau}}
\newcommand{\bomega}{\boldsymbol{\omega}}

\newcommand{\bgamma}{\boldsymbol{\gamma}}
\newcommand{\rhobar}{\overline{\rho}}

\DeclareMathOperator{\NE}{NE}
\DeclareMathOperator{\Div}{Div} % Group of divisors supported on D.

\DeclareMathOperator{\Int}{Int}

\DeclareMathOperator{\SP}{SP}

\DeclareMathOperator{\rk}{rk}

\newcommand{\inc}{\mathrm{in}}
\newcommand{\out}{\mathrm{out}}

\DeclareMathOperator{\Spf}{Spf}

\DeclareMathOperator{\Sp}{Sp} % affinoid spectrum
\DeclareMathOperator{\Spec}{Spec} % affine scheme.
\DeclareMathOperator{\Irr}{Irr} % Set of irreducible components.
\DeclareMathOperator{\Wall}{Wall} % Wall structure.
\DeclareMathOperator{\Sing}{Sing} % Singular locus.

\newcommand{\tk}{\widetilde{k}}

\DeclareMathOperator{\an}{an} % Anaytification
\DeclareMathOperator{\Sk}{Sk} % Skeleton
\newcommand{\Skbar}{\overline{\Sk}} % Compactified skeleton
\newcommand{\ess}{\mathrm{ess}} % essential
\begin{document}
\title[Non-archimedean cylinder counts are logarithmic Gromov-Witten invariants]{Non-archimedean cylinder counts are logarithmic Gromov-Witten invariants}

\author{Thorgal Hinault}
\address{Thorgal Hinault, Department of Mathematics, M/C 253-37, Caltech, 1200 E.\ California Blvd., Pasadena, CA 91125, USA}
\email{thinault@caltech.edu}
\author{Tony Yue YU}
\address{Tony Yue YU, Department of Mathematics, M/C 253-37, Caltech, 1200 E.\ California Blvd., Pasadena, CA 91125, USA}
\email{tyy@caltech.edu}

%\subjclass[2020]{Primary 14N35; Secondary 14J33, 14G22, 14A21}
%\keywords{non-archimedean cylinder counts, spine counts, logarithmic Gromov-Witten invariant, scattering diagram, mirror symmetry}

\date{October 20, 2025}

\maketitle

\begin{abstract}
    We establish a comparison result relating non-archimedean cylinder counts and logarithmic cylinder counts in a smooth affine log Calabi-Yau variety.
    Using the decomposition theorem and the gluing formula from log Gromov-Witten theory, we can express logarithmic cylinder counts in terms of wall type invariants. 
    As a corollary, we show that in the surface case the non-archimedean scattering diagram from Keel-Yu and the logarithmic scattering diagram from Gross-Siebert coincide, and deduce that the two mirror constructions agree. 
    Along the way, we prove the exponential formula, expressing the non-archimedean wall-crossing function as the exponential of a generating series of punctured log Gromov-Witten invariants.
    This provides the first explicit formula relating counts of non-archimedean curves with boundary to punctured log invariants.
\end{abstract}

\personal{Personal comments are shown!}

\tableofcontents

\addtocontents{toc}{\protect\setcounter{tocdepth}{1}}
\section{Introduction}

\personal{Plan of the intro:
\begin{enumerate}
    \item Mirror symmetry and reconstruction problem, SYZ.
    \item SYZ in NA geometry.
    \item Formal models and special fiber. 
    \item Logarithmic construction.
    \item Comparison, scattering diagrams, comparison of mirrors.
\end{enumerate}}

\subsection{Motivations}
Mirror symmetry is one of the most mysterious dualities in mathematics \cite{Hori_Mirror_symmetry}. 
One fundamental problem in mirror symmetry is the construction of mirror manifolds \cite[\S 4]{Cox_Mirror_symmetry_and_algebraic_geometry}.
Explicit constructions of mirror manifolds exist in the toric context via combinatorial dualities \cite{Batyrev_Dual_polyhedra_and_mirror_symmetry,Batyrev_On_Calabi-Yau_complete_intersections}.
In order to go beyond the toric case, the guiding conjecture is the \emph{SYZ conjecture} by Strominger-Yau-Zaslow \cite{Strominger_Mirror_symmetry_is_T-duality}, and its refinements by Gross-Wilson \cite{Gross_Large_complex_structure_limits} and Kontsevich-Soibelman \cite{Kontsevich_Homological_mirror_symmetry}.
Roughly, the conjecture states that a (maximally degenerating) Calabi-Yau manifold admits a torus fibration over a smooth compact base $B$, called the \emph{SYZ fibration}, and that the mirror manifold should be constructed from the counts of holomorphic disks with boundaries on the torus fibers of the SYZ fibration (called \emph{instanton corrections}).

While the existence of the SYZ fibration remains a wide open conjecture in differential geometry, Kontsevich-Soibelman \cite{Kontsevich_Homological_mirror_symmetry} had a remarkable insight that the SYZ fibration resembles the retraction to the skeleton in non-archimedean analytic geometry constructed by Berkovich \cite{Berkovich_Spectral_theory,Berkovich_Smooth_p-adic_analytic_spaces_are_locally_contractible}.
\personal{The skeleton of a non-archimedean analytic space depends on a choice of formal model, but in the Calabi-Yau case it is possible to construct a canonical skeleton referred to as the \emph{essential skeleton} (see \cite{Kontsevich_Affine_structures}).}
Nicaise-Xu-Yu \cite{Nicaise_Xu_Yu_The_non-archimedean_SYZ_fibration} realized the SYZ fibration in the non-archimedean context by establishing that the retraction to the essential skeleton is an affinoid torus fibration outside codimension $2$.

The enumerative geometry for non-archimedean SYZ fibrations has been developed in \cite{Yu_Enumeration_of_holomorphic_cylinders_I,Yu_Enumeration_of_holomorphic_cylinders_II,Keel_Yu_The_Frobenius,Keel-Yu_Log-CY-mirror-symmetry}.
In particular, following the spirit of the SYZ conjecture, the mirror manifold is constructed as the spectrum of a commutative associative algebra with a canonical basis, whose structure constants are given by counts of non-archimedean analytic disks in the non-archimedean SYZ fibration.
In addition, they constructed a canonical scattering diagram to encode the instanton corrections.
It is a collection of walls in the essential skeleton, decorated with \emph{wall-crossing functions}, which are generating series of non-archimedean \emph{infinitesimal cylinder counts}.

Towards the same goal of general mirror construction as in the SYZ conjecture, but based on different technical foundations, the \emph{Gross-Siebert program} (see \cite{Gross_Affine_manifolds,Gross_From_real_affine_geometry,Gross_Mirror_symmetry_via_logarithmic_degeneration_data_I,Gross_Mirror_symmetry_via_logarithmic_degeneration_data_II,Gross_Toric_degenerations_and_Batyrev-Borisov_duality,Gross_Mirror_symmetry_and_the_SYZ}) has motivated the development of logarithmic and punctured log Gromov-Witten theory, culminating in the intrinsic mirror construction of \cite{Gross_Intrinsic_mirror_symmetry} and the canonical logarithmic scattering diagram of \cite{Gross_Canonical-wall-structure-and-intrinsic-mirror-symmetry}.
Via punctured log Gromov-Witten theory, Gross and Siebert are able to define counts analogous to Maslov index $0$ disks, referred to as \emph{wall type invariants}.
The wall-crossing functions in the logarithmic scattering diagram are defined as the exponential of a generating series of wall type invariants.

Both the non-archimedean approach and the logarithmic approach possess their own strengths and advantages.
Comparison of the two enumerative theories can lead to nontrivial results in both fields.
For example, the non-archimedean counts are independent of choice of compactification, while it is not obvious in the logarithmic setting where log curves can have entire components lying in the boundary.
Non-archimedean moduli spaces are often smoother and the resulting counts are more likely to be nonnegative integers instead of rational numbers, which is hard to tell from logarithmic virtual fundamental classes.
On the other hand, logarithmic Gromov-Witten theory is well-established, providing powerful tools such as decomposition and splitting (\cite{Abramovich_Decomposition_degenerate_GW_invariants,Wu_Splitting-log-Gromov-Wittten-Invariants,Gross_Notes-gluing-log-maps}) that are not available in non-archimedean geometry.

\medskip

In this paper, we establish the comparison result relating non-archimedean cylinder counts and logarithmic cylinder counts in a smooth affine log Calabi-Yau variety, see \cref{thm:intro-main-comparison}.
As an application, we deduce the equivalence of the non-archimedean and logarithmic scattering diagrams in the surface case (\cref{thm:intro-exponential-formula}), and hence the equivalence of the non-archimedean and logarithmic mirror constructions (\cref{coro:intro-mirror-algebras}).
The comparison boils down to the comparison of wall-crossing functions, which follows from \cref{prop:intro-infinitesimal-cylinder-counts-to-log-counts}.
Along the way, we prove the exponential formula, expressing the non-archimedean wall-crossing function as the exponential of a generating series of punctured log Gromov-Witten invariants.
This provides the first explicit formula relating counts of non-archimedean curves with boundary to punctured log invariants.
It paves the way for more general comparison results in our subsequent works.

\subsection{Main results}
We now provide a more detailed description of the main results of the paper.

\subsubsection{Non-archimedean cylinder counts are logarithmic Gromov-Witten invariants}
Let $U$ be a connected smooth affine log Calabi-Yau variety over $\bbC$, and $(X,D)$ a simple normal crossing compactification.
We view the pair $(X,D)$ as a log scheme with divisorial log structure.
On the other hand, we can take the Berkovich non-archimedean analytification $U^{\an}$ of $U$ with respect to the trivial valuation on $\bbC$ (see \cite{Berkovich_Spectral_theory}).
We have an essential skeleton $\Sk(U)\subset U^{\an}$, homeomorphic to the subcomplex of the tropicalization $\Sigma(X,D)$ of the log scheme $(X,D)$ spanned by essential divisors.
We now relate the non-archimedean curve counts in $U^{\an}$ and the logarithmic curves counts in $(X,D)$.

On the non-archimedean side, we consider non-archimedean cylinder counts as defined in \cite{Keel-Yu_Log-CY-mirror-symmetry}.
They are refinements of usual $3$-pointed Gromov-Witten invariants obtained by constraining the spines of the stable maps.
The spine is the restriction of the associated tropical curve to the convex hull of the marked points, see also \cite[\S 8]{Keel_Yu_The_Frobenius} for an intrinsic definition.
The precise definition of cylinder spines is given in \cref{def:cylinder-spine}, and a typical cylinder spine is depicted in \cref{fig:cylinder-spine}.
In the end, the \emph{non-archimedean cylinder count} $N(S,A)$ associated to a cylinder spine $S$ and a curve class $A\in \NE (X,\bbZ )$ is defined as the degree of an \'{e}tale map, see \eqref{eq:def-analytic-cylinder-counts}.

The non-archimedean cylinder counts correspond to logarithmic invariants of $(X,D)$ with constrained tropical types. 
We define (decorated) \emph{tropical cylinder types} in $\Sigma(X,D)$ (\cref{def:S-type}), and associate to such a decorated type $\btau$ a \emph{logarithmic cylinder count}
\[N_{\btau} = \frac{\deg[\cM (X,\btau )]^{\vir}}{\vert \Aut (\btau )\vert }. \]
Our first main result is the following comparison theorem, which expresses non-archimedean cylinder counts in terms of logarithmic cylinder counts.

\begin{theorem}[{Main comparison, \cref{thm:main-comparison}}] \label{thm:intro-main-comparison}
    Let $U$ be a connected smooth affine log Calabi-Yau variety satisfying the assumptions of \cref{subsec:log-setup}.
    Let $S$ be a cylinder spine in $\Sk (U)$ of type $\beta = (A,\bu )$.
    Then the non-archimedean cylinder count $N(S,A)$ decomposes as a sum of logarithmic cylinder counts
    \[N(S,A ) = \sum_{\btau} k_{\tau} N_{\btau} , \]
    where the sum is over decorated tropical cylinder types $\btau$ of total curve class $A$, contact orders specified by $\beta$, and spine of type $S$.
\end{theorem}

The theorem is proved using a degenerate point constraint as in \cite{Johnston_Comparison}.
We review this construction and its main properties in \cref{subsec:family-moduli-space}, and obtain a family of point-constrained moduli spaces $\cM_T$ defined over $T = \Spec \bbC\dbb{t}$.
The generic fiber $\cM_{\eta}$ corresponds to a generic point constraint, while on the special fiber $\cM_s$ the point constraint degenerates to a $0$-stratum of the boundary $D$ (see \cref{fig:point-constraint}).
We reformulate the description of tropical types appearing in $\cM_s$ in \cite{Johnston_Comparison} into a decomposition of the virtual fundamental class $[\cM_s]^{\vir}$ indexed by tropical cylinder types in \cref{prop:decomposition-special-fiber}.
While the toric model assumption in \cite{Johnston_Comparison} is not necessary for this decomposition to work, the semi-positivity condition (\cref{assumption:semipositivity}) is crucial to compare with counts in the general fiber.
We then refine this decomposition according to the spine type in \cref{subsubsec:spine-refinement}.
Given a cylinder spine $S$, we restrict the special fiber to the components corresponding to tropical cylinder types with spine type specified by $S$.
The degree of the virtual fundamental class of this stack gives the right hand side of the formula.
Using \cref{lemma:construct-substack}, we then produce a substack of $\cM_T$ and prove that its generic fiber is exactly the space used to define the non-archimedean invariants (see \cref{prop:spine-refinement-generic-fiber}).
We prove this by considering the constant model $X\times T$ of the analytification $X_{\eta}^{\an}$, and use the theory of skeleta for pairs (see \cite{Gubler_Skeletons_and_tropicalizations}) to relate the spine of non-archimedean stable maps in $X_{\eta}^{\an}$ to the tropical type of stable maps to the log scheme $(X_s , D_s)$ (see \cref{prop:spine-tropical-type-special-fiber}). 
The main comparison then follows from the deformation invariance property of the virtual fundamental class.
The proof of \cref{prop:spine-refinement-generic-fiber} relies on the study of the combinatorial properties of tropical cylinder types, which is carried out in \cref{subsubsec:log-cylinder-counts}.

\personal{Our strategy of proof builds on Johnston's work in \cite{Johnston_Comparison}, where a direct comparison of the mirror algebras was achieved by expressing the relevant non-archimedean invariants as the degree of the evaluation map for some moduli space of log stable maps. 
A key element of this comparison is the use of a degenerate point constraint parameterized by the spectrum of a DVR $B = \Spec k^{\circ}$, which constrains the tropical types of log stable maps that appear in the generic fiber.
By considering the constant model $X\times B$ of $X$, we recast this construction into a point constraint for log stable maps in $X\times B$ over $B$, producing a point-constrained moduli space for stable maps in $X\times B$ over $B$ parameterized by $B$.
We then use logarithmic structures to describe the virtual fundamental class of the special fiber of this moduli space.
Using the theory of formal models and skeleta for pairs (see \cite{Gubler_Skeletons_and_tropicalizations}), we relate the skeleton of stable maps in the analytification of the generic fiber to the tropical types of log stable maps in the special fiber, from which we deduce our main comparison.}

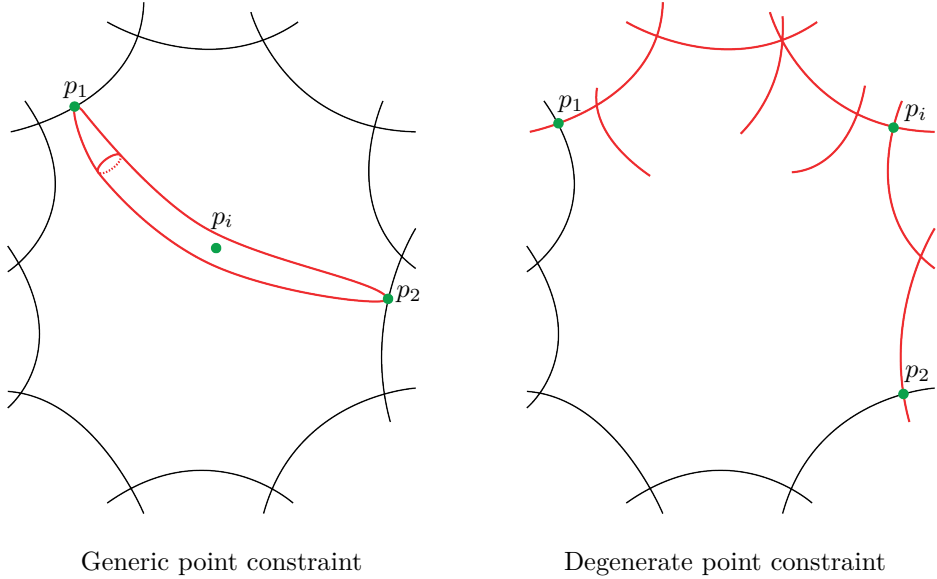
\begin{figure}
    \centering
    \setlength{\unitlength}{0.5\textwidth}
    \begin{picture} (1.65,1)
    %\includesvg[width=0.9\linewidth]{Point-constraint.svg}
    \put(0.0,0.1){\includegraphics[height=.9\unitlength]{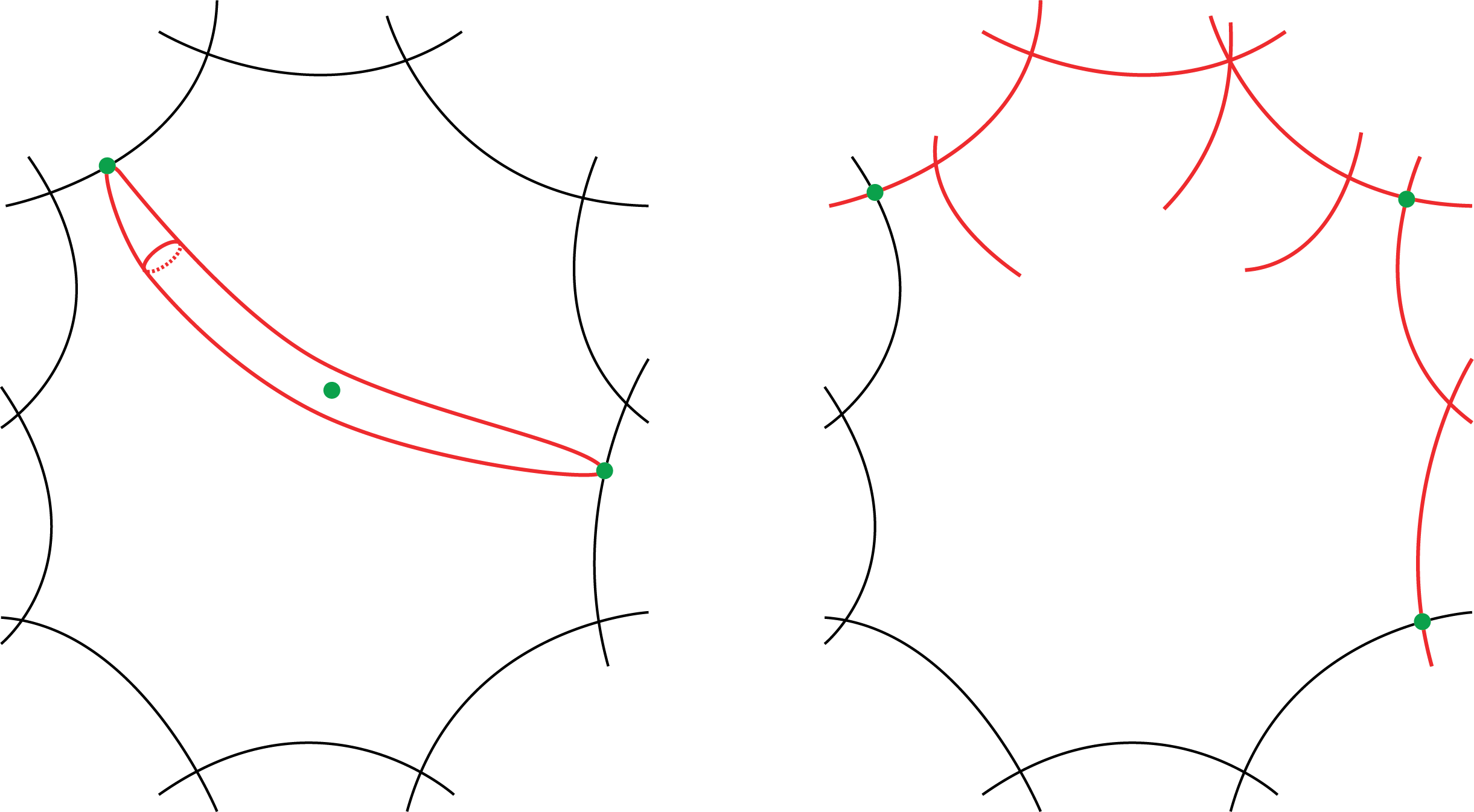}}
    \put(0.1,0.84){$p_1$}
    \put(0.685,.48){$p_2$}
    \put(0.36,0.61){$p_i$}
    \put(0.13,0){Generic point constraint}
    \put(0.97,0.815){$p_1$}
    \put(1.58,0.34){$p_2$}
    \put(1.58,0.8){$p_i$}
    \put(0.98,0){Degenerate point constraint}
    \end{picture}
    \caption{Typical curves appearing on the generic and special fiber of the point constrained moduli spaces constructed in \S \ref{subsec:family-moduli-space}.
    The point constraint is at the marked point $p_i$, while $p_1$ and $p_2$ are mapped to specified components of the boundary $D$.}
    \label{fig:point-constraint}
\end{figure}

\subsubsection{Comparison of mirror symmetry constructions for affine log Calabi-Yau surfaces}

As an application of the comparison of cylinder counts, we compare the non-archimedean and logarithmic scattering diagrams associated to a connected smooth affine log Calabi-Yau surface in \cref{sec:application-scattering-diagram}.
The non-archimedean scattering diagram is constructed using \emph{infinitesimal cylinder counts}, a non-archimedean count of curves with boundary subject to extendability conditions.
On the other hand, the logarithmic scattering diagram is constructed using \emph{wall type invariants}, a special kind of punctured log Gromov-Witten invariants. 
In the surface case, infinitesimal cylinder counts are easily expressed in terms of cylinder counts (corresponding to counts of proper analytic curves).
Applying our main comparison theorem then leads to the following formula.

\begin{proposition}[{\cref{prop:infinitesimal-cylinder-counts-to-log-counts}}] \label{prop:intro-infinitesimal-cylinder-counts-to-log-counts}
    Let $U$ be a connected smooth affine log Calabi-Yau surface with a compactification $(X,D)$ as in \cref{assumption:compactification}.
    Let $S$ be an infinitesimal cylinder spine in $\Sk (U)$ of type $\beta =(A,\bu )$ with $A\in\NE (X,\bbZ )$ and $\bu = (\bu_1 ,\bu_2, 0)$.
    Then 
    \[N(S,A) =  \sum_{\ell\geq 0} \sum_{\mu_1,\dots ,\mu_{\ell}\geq 0}\sum_{\substack{\btau_1,\dots ,\btau_{\ell}  }} \frac{\prod_{i=1}^{\ell}  \big( k_{\tau_i} W_{\btau_i}\big)^{\mu_i}}{\mu_1!\cdots \mu_{\ell}!}, \]
    where the last sum is over decorated wall types $\btau_1,\dots ,\btau_{\ell}$ with contact order $u_{\tau_i}$ and total curve class $A_i$ that satisfy: 
    \begin{align*}
        \mu_1 u_{\tau_1} + \cdots +\mu_{\ell} u_{\tau_{\ell}}  &= \bu_1 +\bu_2 , \\
        \mu_1 A_1 + \cdots + \mu_{\ell} A_{\ell} &= A. 
    \end{align*}
\end{proposition}

In the surface case, the non-archimedean and logarithmic scattering diagrams are determined by the wall-crossing functions at points $x\in \Sk (U)\setminus 0$.
The \emph{non-archimedean wall-crossing function} at $x$ is the following generating series of infinitesimal cylinder counts (see \eqref{eq:wall-crossing-function-NA})
\begin{equation*}
    f_x^{\an} (t,z) = \sum_{\substack{\bw\in n^{\perp}\\ A\in \NE (X,\bbZ)}} N(S,A) t^A z^{\bw} ,
\end{equation*}
where $n$ is a fixed normal vector to the ray from $0$ to $x$, $S$ is any infinitesimal cylinder spine with contact orders $(\bu_1,\bu_2 , 0)$ such that $\bu_1+\bu_2=\bw$ and $\langle n ,\bu_1\rangle = 1$, and with the bending vertex mapped to $x$.
The \emph{logarithmic wall-crossing function} at $x$ is obtained as the exponential of a generating series of wall types invariants (see \eqref{eq:wall-crossing-function-log})
\begin{equation*}
    f_x^{\log} (t,z)= \exp\bigg( \sum_{\btau} k_{\tau} W_{\btau} t^A z^{-u_{\tau}} \bigg),
\end{equation*}
where the sum is over decorated wall types $\btau$ with total curve class $A$ and contact order $u_{\tau}$ at the unique leg $L_{\out}$, which satisfy $x\in h_{\tau} (\tau_{\out} )$.

From \cref{prop:intro-infinitesimal-cylinder-counts-to-log-counts} we deduce the second main theorem, which provides the equivalence of the two scattering diagrams.
We call it the \emph{exponential formula}, because it expresses the non-archimedean wall-crossing function as the exponential of a generating series of wall type invariants.

\begin{theorem}[{Exponential formula, \cref{thm:exponential-formula}}] \label{thm:intro-exponential-formula}
    Let $U$ be a connected smooth affine log Calabi-Yau surface.
    For any $x\in \Sk (U)\setminus 0$, the logarithmic and non-archimedean wall-crossing functions at $x$ agree: 
    %\[f_x^{\an} (t,z) = f_x^{\log} (t,z) .\]
    \[\sum_{\substack{\bw\in n^{\perp}\\ A\in \NE (X,\bbZ)}} N(S,A) t^A z^{\bw} = \exp\bigg( \sum_{\btau} k_{\tau} W_{\btau} t^A z^{-u_{\tau}} \bigg).\]
    Consequently, the non-archimedean scattering diagram from \cite{Keel-Yu_Log-CY-mirror-symmetry} and the logarithmic scattering diagram from \cite{Gross_Canonical-wall-structure-and-intrinsic-mirror-symmetry} coincide.
    In particular, for any curve class $A\in \NE( X,\bbZ )$, we have $\Wall_A^{\an} = \Wall_A^{\log}$.
\end{theorem}

Since the scattering diagrams determine the mirror algebra through their theta functions, we deduce as a corollary that the non-archimedean and logarithmic mirror constructions agree. 

\begin{corollary}[{\cref{coro:mirror-algebras}}] \label{coro:intro-mirror-algebras}
    Let $U$ be a connected smooth affine log Calabi-Yau surface.
    Then the non-archimedean and logarithmic mirror constructions provide the same mirror algebra.
\end{corollary}

\subsection{Future directions}
We indicate future directions of research which are natural extensions of the present paper.
We intend to study these in subsequent works.
\begin{enumerate}[wide]
    \item Consider curves with more than two contact points with the boundary.
    These invariants require a choice of constraint for the domain modulus.
    Johnston \cite{Johnston_Comparison} studied the comparison without spine refinements.
    We are working on a more direct comparison with specified spine types.
    \item Consider non-archimedean curves with boundaries and punctured log curves.
    We need to impose generalized tail conditions as in \cite{Keel-Yu_Log-CY-mirror-symmetry} in the non-archimedean setting and relate them to logarithmic constructions.
    This would lead to a direct comparison of the non-archimedean mirror construction and the logarithmic mirror construction, as we did in the current paper for the two-dimensional case.
    \item Extend the comparison to allow virtual fundamental classes in the non-archimedean setting.
    % \item Develop the theory of formal models for non-archimedean stable maps to pairs and remove the semi-positivity assumption placed on the boundary.
    % We hope that by developing a general theory of formal models for non-archimedean stable maps to $(X,D)$, more general comparison of virtual fundamental classes can be obtained.     
\end{enumerate}

\subsection*{Acknowledgements}
We would like to thank Dan Abramovich, H\"{u}lya Arg\"{u}z, Pierrick Bousseau, Qile Chen, Tom Graber, Mark Gross, Zhaoxing Gu, Sam Johnston, Soham Karwa, Maxim Kontsevich, Sean Keel, Bernd Siebert, Yixian Wu, Weihong Xu, Chi Zhang, Shaowu Zhang \todo{insert names} for valuable discussions.
The authors were partially supported by NSF grants DMS-2302095 and DMS-2245099.

\addtocontents{toc}{\protect\setcounter{tocdepth}{2}}

\section{Non-archimedean and logarithmic cylinders} \label{sec:preliminaries}

Let $\bbk$ be an algebraically closed field of characteristic $0$.
Let $k^{\circ}$ be a noetherian $\bbk$-algebra which is a complete DVR of height $1$.
We denote by $k$ the fraction field of $k^{\circ}$, and by $\tk$ its residue field.
We let $B\coloneqq \Spec k^{\circ}$, and denote by $s = \Spec \tk$ (resp. $\eta = \Spec k$) its closed point (resp. generic point).
\personal{Concrete case: $\bbk =\bbC$, $k^{\circ} = \bbC\dbb{t}$, $k = \bbC\dbp{t}$, $\tk  =\bbC$.} \\

In this section, after formulating the assumptions on the log Calabi-Yau pair $(X,D)$, we use the constant model $X\times\Spec \bbk$ for $X_{\eta}^{\an}$ to relate the spine of non-archimedean stable maps in $X_{\eta}^{\an}$ to the tropical type of log stable maps in $X_s$, equipped with the divisorial log structure coming from $D_s\subset X_s$ (\cref{prop:spine-tropical-type-special-fiber}).
In \S\ref{subsec:walls-balancing} we recall the construction of the non-archimedean wall structure considered in \cite{Keel-Yu_Log-CY-mirror-symmetry}, and construct its logarithmic analogue using wall types (\cref{constr:logarithmic-walls}).
In \S\ref{subsec:cylinder-counts}, we review the construction of non-archimedean cylinder counts, and proceed to define their logarithmic counterparts after defining tropical cylinder types in \cref{def:S-type}.
We study the deformation and transversality properties of tropical cylinder types, and prove that a certain sum of logarithmic cylinder counts is birationally invariant under log-\'{e}tale modifications in \cref{prop:birational-invariance-log-cylinders}. 
All of these results are needed for the main comparison in \S\ref{sec:comparison-counts}.

\subsection{Affine log Calabi-Yau pair} \label{subsec:log-setup}
In this subsection, we state the assumptions on the log Calabi-Yau pair $(X,D)$, and compare logarithmic stable maps in $(X,D)$ with non-archimedean stable maps in $X_{\eta}^{\an}$ using a constant model for $X$ (see \cref{prop:spine-tropical-type-special-fiber}).
\subsubsection{Assumptions}

\begin{definition}[Log Calabi-Yau variety]
    A variety $U$ over $\bbk$ is a \emph{log Calabi-Yau variety} if there exists an snc compactification $U\subset X$ with boundary $D\coloneqq X\setminus U$, such that $K_X+D$ is numerically equivalent to an effective $\bbQ$-divisor supported on $D$.
    The pair $(X,D)$ is called a \emph{log Calabi-Yau pair}.
\end{definition}

For the rest of the paper, we fix a $d$-dimensional connected smooth affine log Calabi-Yau variety $U$ over $\bbk$ and a log Calabi-Yau pair $(X,D)$ which compactifies $U$.
We write $D = \sum_{i\in I_D} D_i$, and fix a choice of rational numbers $a_i\geq 0$ such that
\begin{equation}\label{eq:log-CY-canonical-class}
    K_X + D = \sum_{i\in I_D} a_i D_i.
\end{equation}
The \emph{essential part of the boundary} $D^{\ess}\subset D$ is defined as
\[D^{\ess}\coloneqq \sum_{i\colon a_i = 0} D_i ,\]
and we denote by $I_{D^{\ess}}\subset I_D$ the indices for which $a_i = 0$.
The stratum associated to a subset $I\subset I_D$ is
\[D_I\coloneqq \bigcap_{i\in I} D_i .\]

Throughout the paper, we make the following assumptions on $U$ and $(X,D)$, following \cite{Gross_Canonical-wall-structure-and-intrinsic-mirror-symmetry,Keel-Yu_Log-CY-mirror-symmetry}.

\begin{assumption}[Minimal model] \label{assumption:minimal-model}
    There exists a projective compactification $U\subset X'$ with boundary $D' = U\setminus X'$, such that $K_{X'} + D'$ is divisorial log terminal (dlt) and trivial.
\end{assumption}

\begin{assumption} \label{assumption:compactification}
We assume the following on the compactification $(X,D)$.
\begin{enumerate}[wide]
    \item (Log Kodaira dimension $0$) \label{assumption:log-kodaira-dimension}
    The right-hand side of \eqref{eq:log-CY-canonical-class} is integral, with Kodaira dimension $0$.
    \item (Maximality)\label{assumption:maximal-essential-skeleton}
    The essential part of the boundary $D^{\ess}$ contains a $0$-dimensional stratum.
    \item (Almost minimality) \label{assumption:almost-minimal}
    The pair $(X,D)$ is almost minimal, in the sense that every $1$-dimensional stratum of $D^{\ess}$ only intersects irreducible components of $D^{\ess}$.
    \item (Connectedness)  \label{assumption:connected-essential-skeleton} For every subset $I\subset I_{D^{\ess}}$ of cardinality at most $d-2$, the following divisor on $D_I$ is connected:
    \[\sum_{j\in I_{D^{\ess}}\setminus I}D_j \cap D_I.\]
\end{enumerate}
\end{assumption}

Finally, we make the following semi-positivity assumption which will allow us to relate logarithmic stable maps to relative stable maps.

\begin{assumption}[Semi-positivity] \label{assumption:semipositivity}
    There exists a nef divisor $D'$ supported on $D$, i.e. $D' =\sum_{i\in I_D} b_i D_i$ with some $b_i >0$.
\end{assumption}

\begin{remark}
    We comment on the relevance of each assumption.
    \begin{enumerate}[wide]
        \item \cref{assumption:compactification}(1)-(2) imply that $U$ has a unique regular volume form $\omega$, with at worst simple poles, and that $D^{\ess}$ is the polar locus. 
        \item Under \cref{assumption:minimal-model} and \cref{assumption:compactification}(1)-(2), we can always find a log Calabi-Yau pair $(X,D)$ which compactifies $U$ and satisfies \cref{assumption:compactification}(3) (see \cite[Proposition 2.12]{Keel-Yu_Log-CY-mirror-symmetry}).
        \item \cref{assumption:compactification}(2)-(4) correspond to Assumption 1.1 of \cite{Gross_Canonical-wall-structure-and-intrinsic-mirror-symmetry}, which is required to define the canonical wall structure through log Gromov-Witten theory.
        \item Conjecturally, \cref{assumption:minimal-model} always holds, and \cref{assumption:compactification}(1) implies \cref{assumption:compactification}(2)-(4).
        \item \cref{assumption:semipositivity} is used to apply the arguments of \cite{Johnston_Comparison}, and rule out the possibility of log stable maps in $(X,D)$ mapping unmarked components to $D$.
        It is always satisfied in the surface case by \cite[Lemma 6.9]{Gross_Mirror_symmetry_for_log_Calabi-Yau_surfaces_I_published}, since we assume $U$ affine.
    \end{enumerate}
\end{remark}

We introduce the group $\Div_D (X)\coloneqq \bigoplus_{i\in I_D} \bbZ[D_i]$, and its dual $\Div_D (X)^{\ast}$ with basis $[D_i]^{\ast}$.
The dual intersection complex of $(X,D)$ is the cone complex
\[\Sigma \coloneqq  \bigcup_{I\subset I_D} \bigg\lbrace \sum_{i\in I} t_i [D_i]^{\ast}\mid t_i\geq 0 \text{ and }D_I\neq \emptyset\bigg\rbrace \subset \Div_D (X)^{\ast}\otimes_{\bbZ}\bbR\simeq \bbR^{I_D} ,\]
with the natural face inclusions.
Here, we view each cone as a cone inside $\Div_D (X)^{\ast}\otimes_{\bbZ}\bbR$ and identify $\Sigma$ with its support $\vert \Sigma\vert$.
We denote by $\Sigma^{\ess}\subset \Sigma$ the subcone complex spanned by components of $D^{\ess}$, its support is contained in $\bbR^{I_{D^{\ess}}}\subset \bbR^{I_D}$.

We recall that if we equip $X$ with the divisorial log structure coming from $D$, then the tropicalization $\Sigma (X ,D)$ is identified with $\Sigma$ as a cone complex.
On the other hand, if we equip $X$ with the divisorial log structure coming from $D^{\ess}$, then $\Sigma (X, D^{\ess})$ is identified with $\Sigma^{\ess}$.
The inclusion $D^{\ess} \subset D$ induces a morphism of log schemes 
\[(X,D^{\ess})\longrightarrow (X, D )\]
which tropicalizes to the inclusion $\Sigma^{\ess}\subset \Sigma$.

We introduce the base-changed varieties
\begin{align*}
    X_B \coloneqq X\times_{\bbk} B ,\;
    D_B\coloneqq D\times_{\bbk} B , \; U_B \coloneqq U\times_{\bbk} B
\end{align*}
which will be used as a bridge between the analytic geometry of the analytification $X_{\eta}^{\an}$, and the logarithmic geometry of $(X_s, D_s)$.

\subsubsection{Skeleton and essential skeleton}
The completion of the pair $(X_B,D_B)$ along the special fiber produces a formal strictly semistable pair $(\fX_B,\fD_B )$ in the sense of \cite{Gubler_Skeletons_and_tropicalizations}.
Since $X_{\eta}$ is proper, the analytic generic fiber of $\fX_B$ coincides with the analytication $X_{\eta}^{\an}$.
This pair produces a \emph{skeleton} $\Sk (X,D)\subset U_{\eta}^{\an}$ and a retraction 
\begin{equation} \label{eq:retraction-skeleton}
    \rho\colon U_{\eta}^{\an}\longrightarrow \Sk (X,D) .
\end{equation}
Taking the closure of $\Sk (X,D)$ in $X_{\eta}^{\an}$ produces the compactified skeleton $\Skbar (X,D)\subset X_{\eta}^{\an}$.
The retraction $\rho$ extends to $\rhobar\colon X_{\eta}^{\an}\rightarrow \Skbar (X,D)$.
The skeleton $\Sk (X,D)$ is naturally identified with (the support of) the dual intersection complex $\vert \Sigma\vert$.

The \emph{essential skeleton} $\Sk (U)\subset U_{\eta}^{\an}$ is defined intrinsically as the maximum locus of the Temkin seminorm $\Vert\omega\Vert\colon U_{\eta}^{\an}\rightarrow \bbR_{\geq 0}$ associated to the volume form $\omega$ on $U_{\eta}^{\an}$.
It is included in $\Sk (X,D)$, and corresponds to the subcomplex $\Sigma^{\ess}\subset \Sigma$ under the identification $\Sk (X,D)\simeq\vert \Sigma\vert$.
Using those identifications, the projection $\bbR^{I_D}\rightarrow \bbR^{I_{D^{\ess}}}$ produces a diagram
\begin{equation} \label{cd:essential-skeleton}
    \begin{tikzcd}
        U_{\eta}^{\an}\ar[r,"\rho"]\ar[rd, "\rho^{\ess}"] & \Sk(X,D)\ar[d, shift right = 1ex] \\
        & \Sk (U) \ar[u,hook, shift right= 1ex]
    \end{tikzcd}
\end{equation}
where only the inner triangle is commutative.
We also have the compactified versions $\Skbar^{\ess} (X,D)\subset X_{\eta}^{\an}$ and $\rhobar^{\ess}\colon X_{\eta}^{\an}\rightarrow \Skbar^{\ess} (X,D)$, which depend on the compactification.

\personal{We denote by $\mathsf{B}^{\ess}\subset \mathsf{B}$ the support of $\Sigma^{\ess}$.}

We briefly discuss the integral affine structure on the skeleton, or equivalently on $\Sigma$.
The lattice $\Div_D (X)^{\ast}$ induces a piecewise linear integral manifold structure on $\Sk (X,D)$, and in particular there is a well-defined integral tangent space $T\sigma$ to any cone $\sigma\in \Sigma$.
If $\sigma$ corresponds to a subset $I_{\sigma}\subset I_D$, then there is an identification 
\begin{equation}\label{eq:naive-tangent-lattice}
    T\sigma\simeq \bbZ^{I_{\sigma}}\subset \bbZ^{I_D}.
\end{equation}

Under \cref{assumption:minimal-model} and the maximal boundary condition on $(X,D)$ (\cref{assumption:compactification}(2)), we can use the compactification $U\subset X$ to produce an affine structure on 
\begin{equation} \label{eq:remove-codim-2-locus}
    \Sigma_0\coloneqq \Sigma^{\ess}\setminus \bigcup_{\substack{\sigma\in\Sigma \\\dim \sigma\leq d-2}} \sigma
\end{equation}
which extends the affine structure on the maximal cones in $\Sigma$ (see \cite[\S 1.3]{Gross_Canonical-wall-structure-and-intrinsic-mirror-symmetry} and \cite[\S 2.2]{Keel-Yu_Log-CY-mirror-symmetry}).
When we discuss the balancing condition for tropical curves in $\Sigma$ (or in $\Sk (X,D)$), we will only impose conditions on vertices mapping to $\Sigma_0$ using this extended affine structure.
We denote by $\Sing (\Sk (U))\subset \Sk (U)$ the locus corresponding to $\Sigma^{\ess}\setminus \Sigma_0$.

\personal{We denote by $\mathsf{B}$ the support of $\Sigma$.}

\subsubsection{Logarithmic stable maps in $X_B/B$}

To prove the main comparison result (\cref{thm:main-comparison}), which expresses non-archimedean counts in $X_{\eta}^{\an}$ in terms of logarithmic Gromov-Witten invariants of $X_s$, we study log stable maps in $X_B/B$, where $X_B$ is equipped with the divisorial log structure coming from $D_B$.
The non-archimedean counts rely on the notion of spine.

\begin{definition}[Spine of a non-archimedean stable map] \label{def:analytic-spine}
    Let $f\colon (C, (p_1,\dots, p_n))\rightarrow X_{\eta}^{\an}$ be an analytic stable map, with $C$ a proper analytic curve.
    If $f^{-1} (D_{\eta}^{\an} )$ is supported on the marked points, we define the \emph{spine of $f$} as the $\Sigma^{\ess}$-piecewise linear map
    \[\overline{\Sp} (f)\colon \Gamma^s\longrightarrow \Skbar^{\ess} (X,D)\]
    obtained by restricting $f$ to $\Gamma^s\subset C$, the convex hull of the marked points, and by composing with the retraction $\rhobar^{\ess}$.
    We also refer to the following map as the spine of $f$:
    \[\Sp (f)\colon \Gamma^s\setminus \lbrace p_i\rbrace_{1\leq i\leq n}\longrightarrow \Sk (U) .\]
\end{definition}

Consider an algebraic stable map in $X_B$ relative to $B$:
\begin{equation} \label{cd:stable-map-over-B}
\begin{tikzcd}
    (C_B , (p_1,\dots, p_n))\ar[r,"f"] \ar[rd] & X_B \ar[d] \\
    & B.
\end{tikzcd}
\end{equation}
The completion of $f$ along the special fiber produces a formal model for the analytification of the generic fiber $f_{\eta}^{\an}$.
If $f_{\eta}^{-1} (D_{\eta} )$ is supported on the marked points, then by functoriality of the skeleton (see \cite{Gubler_Skeletons_and_tropicalizations}) we obtain an induced map between skeleta
\[\overline{h}\colon (\Gamma ,(v_1,\dots, v_n))\longrightarrow \Skbar (X,D) ,\]
where $\Gamma\subset C_{\eta}^{\an}$ is the compactified skeleton associated to the strictly semistable pair $(C_B, (p_1,\dots, p_n))$.
Here, $\Gamma$ is a metric graph and $(v_1,\dots, v_n)$ are distinguished $1$-valent vertices whose adjacent edges have infinite length.
We denote by $h\colon (\Gamma_{\circ} , (v_1,\dots, v_n))\rightarrow \vert\Sigma\vert\simeq \Sk (X,D)$ the uncompactified version, where $\Gamma_{\circ} = \overline{h}^{-1} (\Sk (X,D))$.
We refer to infinite edges of $\Gamma_{\circ}$ with endpoint removed as legs.
The map $h$ is $\Sigma$-piecewise linear, and we have the following description of the slope along an edge or a leg of $\Gamma_{\circ}$.

\begin{lemma} \label{lemma:contact-order-skeleton-map}
  Let $f\colon (C_B,(p_1,\dots p_n))\rightarrow X_B$ be a stable map as in \eqref{cd:stable-map-over-B}, such that $f_{\eta}^{-1} (D_{\eta})$ is supported on the marked points.
  Let $h\colon (\Gamma_{\circ}, (v_1,\dots , v_n))\rightarrow \Sigma$ denote the associated piecewise linear map.
  Let $v$ be a vertex in $\Gamma_{\circ}$  mapping to the relative interior of a cone $\sigma \subset \Sigma$.
  We identify the tangent lattice $T\sigma$ with a sublattice of $\Div_D (X)^{\ast}\simeq \bbZ^{I_D}$ as in \eqref{eq:naive-tangent-lattice}.
  \begin{enumerate}[wide]
    \item Let $\ell$ be a leg adjacent to $v$, oriented away from $v$, corresponding to a marked point $p$.
    Then $d_{\ell} h\in T\sigma\subset \bbZ^{I_D}$ has $i$-th component given by the contact order of $f$ at $p$ along $D_i$.
    \item Let $e$ be an edge adjacent to $v$, oriented away from $v$, corresponding to a node $q\in C_s$.
    Then $d_e h\in T\sigma\subset \bbZ^{I_D}$ has $i$-th component given by the contact order of $f$ at $q$ along $D_i$.
  \end{enumerate}
\end{lemma}

\begin{proof}
  This follows from the explicit description of the induced map on skeleta in \cite[\S 5]{Gubler_Skeletons_and_tropicalizations}.
  More precisely, in case (1) we can find open neighbourhoods $U\subset X_B$ and $V\subset C_B$ of $f(p)$ and $p$ such that $f(V)\subset U$, and \'{e}tale maps
  \begin{align*}
    \psi_V &\colon V\longrightarrow \Spf k^{\circ}\langle x_0,x_1\rangle / (x_0 -t) ,\\
    \psi_U &\colon U\longrightarrow \Spf k^{\circ} \langle y_0,\dots, y_d \rangle / (y_0 -t) ,
  \end{align*}
  such that the marked point $p$ is given by the divisor $\psi_V^{\ast}(x_1)$ and $D\vert_{U}$ is given by the functions $\psi_U^{\ast} (y_1),\dots ,\psi_U^{\ast} (y_r )$.
  For $0\leq i\leq r$, let $f_i\coloneqq f_{\eta}^{\ast}(\psi_U)_{\eta}^{\ast} (y_i)$.
  The restriction of $h$ to $\ell$ is given by 
  \begin{equation}\label{eq:local-form-map-skeleta}
      (-\log \vert f_0\vert ,\dots, -\log \vert f_r\vert )\colon \ell\longrightarrow \sigma .
  \end{equation}
  In the affine charts $\ell\simeq \bbR_+$ and $\sigma\simeq \bbR_+^{I_{\sigma}}\subset \bbR^{I_D}$, the component $-\log \vert f_i\vert$ corresponds to $t\mapsto h(v) + n_it$, where $n_i$ is the order of vanishing of $f_i$ at $p$ \cite[(5.3.1)]{Gubler_Skeletons_and_tropicalizations}.
  This proves (1).

  For the case of a node, the target of $\psi_V$ is $\Spf k^{\circ} \langle x_0, x_1\rangle /(x_0x_1 - t)$.
  The map $h$ restricted to $e$ has the same components as \eqref{eq:local-form-map-skeleta}.
  An affine chart for the edge $e$ is given by $e\simeq \lbrace (t,1-t) , t\in [0,1]\rbrace\subset \bbR_+^2$.
  Let $C_{s,1}$ and $C_{s,2}$ denote the two irreducible components of $C_s$ containing $p$.
  The component $-\log \vert f_i\vert$ corresponds to $t\mapsto n_{i,1} t + n_{i,2}(1-t)$, where $n_{i,1}$ (resp. $n_{i,2}$) is the order of vanishing of $f_i$ along $C_{s,1}$ (resp. $C_{s,2}$).
  Then the slope of $h$ in the direction of $D_i$ is $n_{i,1}-n_{i,2}$, which equals the contact order of $f$ at $q$ along $D_i$.
  This proves (2).
\end{proof}

Under the identification $\Sk (X,D)\simeq \Sigma$, the skeleton of an analytic stable map produces a tropical type $\tau_f = (\Gamma_{\circ} , \bsigma_f, \bu_f)$ to $\Sigma \simeq \Sigma (X_s,D_s)$.
In this notation, $\Gamma_{\circ}$ is a combinatorial graph (we forget the metric), $\bsigma_f$ associates to each vertex, edge and leg its image cone in $\Sigma$, and $\bu_f$ records the slope of $h$ along each edge and leg. 
We establish some language regarding types of tropical maps to $\Sigma$.

\begin{definition}[Simple type, spine type]\label{def:spine-type}
    Let $\tau = (G,\bsigma, \bu )$ be a tropical type in $\Sigma$.
    \begin{enumerate}[wide]
        \item A vertex $v\in V(G)$ is called a \emph{bending vertex} if $\sum_{e\ni v} \bu (e)\neq 0$, where all the edges are oriented away from $v$ and the sum is taken in $\bbZ^{I_D}$.
        \item A vertex $v\in V(G)$ is called \emph{redundant} if it is $2$-valent, not a bending vertex, and $\bsigma (e) = \bsigma (e')$, where $e,e'\in E(G)\cup L(G)$ denote the edges or legs adjacent to $v$.
        \item The type $\tau$ is \emph{simple} if it has no redundant vertices.
        The \emph{simplification of $\tau$} is the type obtained by removing redundant $2$-valent vertices.
        \item The \emph{spine type of $\tau$}, denoted by $\Sp (\tau )$, is the simplification of the type obtained by restricting $\tau$ to the convex hull of $L(G)$ in $G$.
    \end{enumerate}
\end{definition}

Given a family of log stable maps in $X_B/B$, the following proposition provides a link between the spine of the induced analytic stable map in $X_{\eta}^{\an}$ and the spine type of the induced log stable map in $X_s$.
The assumption on the spine type of the special fiber is necessary because the spine of an analytic stable map is defined using $\rho^{\ess}$ (Eq. \eqref{cd:essential-skeleton}) which includes the projection to $\Sk (U)$.

\begin{proposition} \label{prop:spine-tropical-type-special-fiber}
  Let $f\colon (C_B, (p_1,\dots, p_n))\rightarrow X_B$ be a log stable map in $X_B/B$ as in \eqref{cd:stable-map-over-B}, denote by $\tau = (G,\bsigma,\bu)$ the tropical type in $\Sigma$ induced by $f_s$.
  Assume that $f_{\eta}^{-1} (D_{\eta} )$ is supported on the marked points, and that $\Sp (\tau)$ is induced by a tropical type in $\Sigma^{\ess}$.
  Then the type $\Sp (\tau)$ is given by $\Sp (f_{\eta}^{\an} )$.
\end{proposition}

\begin{proof}
    Since $f_{\eta}^{-1} (D_{\eta} )$ is supported on the marked points, by functoriality of the skeleton we obtain a $\Sigma$-piecewise linear map 
    \[h\colon (\Gamma_{\circ} ,(v_1,\dots, v_n))\longrightarrow \Sigma \]
    producing a type $\tau_f = (\Gamma_{\circ} , \bsigma_f ,\bu_f)$ to $\Sigma$.
    By construction of the skeleton, the graph $\Gamma_{\circ}$ corresponds to the dual intersection complex of $C_s$ together with infinite legs corresponding to marked points.
    Hence $\Gamma_{\circ}  =G$ as combinatorial graphs. 
    By construction of the induced map on skeleta, we also have $\bsigma = \bsigma_f$.
    Finally, \cref{lemma:contact-order-skeleton-map} implies that $\bu = \bu_f$.
    We deduce that $\tau  =\tau_f$ as tropical types to $\Sigma$.

    It remains to prove that $\Sp (\tau_f)$ coincides with the simplification of the type associated to $\Sp (f_{\eta}^{\an})$.
    Let $\Gamma_{\circ}^s$ denote the convex hull of the marked points in $\Gamma_{\circ}\subset C_{\eta}^{\an}$, it is the domain of $\Sp (f_{\eta}^{\an} )$.
    The assumption on $\tau$ and the previous paragraph imply that $\bsigma_f (x)\in \Sigma^{\ess}$ for all $x\in V(\Gamma_{\circ}^s)\cup E(\Gamma_{\circ}^s) \cup L(\Gamma_{\circ}^s)$.
    Thus the restriction of $f_{\eta}^{\an}$ to $\Gamma_{\circ}^s$ factors through the essential skeleton, and hence composing with the projection $\Sigma\rightarrow \Sigma^{\ess}$ does not change the type of $\Sp (\tau_f)$.
    This concludes the proof.
\end{proof}

Let $\beta = (A, \bu )$ denote the data of a curve class $A\in \NE (X,\bbZ )$ and a tuple of divisorial contact orders $\bu = (\bu_1,\dots, \bu_n)$ with divisorial centers in $D^{\ess}$, by which we mean that $\bu_i\in T\sigma$ for a cone $\sigma\in \Sigma^{\ess}$.
By taking $\sigma$ to be minimal we obtain a unique decomposition
\begin{equation}\label{eq:contact-orders}
    \bu_i = \sum_j m_{ij} v_{ij},
\end{equation}
where $\lbrace v_{ij}\rbrace_j$ are primitive generators of the rays of $\sigma$.

In order to define non-archimedean counts, we introduce the following open substacks of the moduli stack of \emph{relative} stable maps in $(X_{\eta},D_{\eta})$, which does not involve any log structure.
We note that since the domain curves have no automorphisms, those stacks are in fact schemes.

\begin{definition}
    The moduli space $\cM (U_{\eta},\beta )$ is the stack of relative stable maps $f\colon (C, (p_1,\dots, p_n))\allowbreak\rightarrow X_{\eta}$ to $(X_{\eta},D_{\eta})$ of type $\beta$ satisfying the following conditions:
    \begin{enumerate}[wide]
        \item for each $i\in \lbrace 1,\dots, n\rbrace$ such that $\bu_i\neq 0$, and each component $D_{\eta, j}\subset D_{\eta}$ corresponding to $v_{ij}$ in \eqref{eq:contact-orders}, we require that $f^{-1} (D_{\eta ,j})$ vanish scheme theoretically at $p_i$ to order at least $m_{ij}$, and $\deg f^{\ast} (D_{\eta, j}) =m_{ij}$,
        \item no irreducible component of $C$ maps into $D_{\eta}$, and
        \item the domain curve $(C,(p_1,\dots, p_n))$ is stable.
    \end{enumerate}
\end{definition}

    For $[f\colon (C,(p_1,\dots, p_n))\rightarrow X_{\eta}]\in \cM (X_{\eta} ,\beta )$, keeping the notations of (2) in the previous definition, we have $f^{-1} (D_{\eta , j}) =m_{ij}p_i$ scheme theoretically.
    In particular we can define $\Sp (f^{\an} )$ as in \cref{def:analytic-spine}.

\begin{definition}
    The moduli space $\cM^{\sm} (U_{\eta},\beta )\subset \cM (U_{\eta} ,\beta )$ is the open subscheme of relative stable maps $f\colon (C, (p_1,\dots, p_n))\rightarrow X_{\eta}$ to $(X_{\eta},D_{\eta})$ such that the pullback bundle $f^{\ast} T_{X_{\eta}} (-\log D_{\eta} )$ is trivial.
\end{definition}

The relevance of this space to define non-archimedean counts lies in the following proposition.

\begin{proposition}[{\cite[Lemma 4.24]{Keel-Yu_Log-CY-mirror-symmetry}}]
    Let $i\in\lbrace 1,\dots, n\rbrace$ such that $\bu_i = 0$.
    Then the map
    \[\Phi_i\coloneqq (\dom ,\ev_i)\colon \cM^{\sm} (U_{\eta},\beta )\longrightarrow \overline{\cM}_{0,n}\times X_{\eta}\]
    taking the domain and evaluating at $p_i$ is \'{e}tale.
\end{proposition}

We denote by $\cM (X_{\eta} ,\beta )$ the moduli space of log stable maps of type $\beta$ to $X_{\eta}$.
To relate non-archimedean counts to logarithmic counts, we will use the following embedding.

\begin{proposition}[{\cite[Lemma 5.2]{Johnston_Comparison}}]\label{prop:embedding-smooth-locus}
    There is an open embedding $\cM^{\sm} (U_{\eta},\beta )\subset \cM (X_{\eta} , \beta )$.
\end{proposition}

\subsection{Walls and balancing condition} \label{subsec:walls-balancing}
In this subsection, we define non-archimedean and logarithmic walls, see \cref{constr:analytic-walls} and \cref{constr:logarithmic-walls} respectively. 
We define the notion of spine and transverse spine with respect to a set of walls.
We also review the definition of wall types invariants and broken line invariants, and establish the transversality and balancing properties for broken line types (\cref{lemma:transversality-broken-lines,lemma:balancing-condition-broken-line}).

We introduce the following notion of wall, which contains strictly less data than the wall structures considered in \cite{Gross_Canonical-wall-structure-and-intrinsic-mirror-symmetry}.
At this stage, instead of decorating the support of walls with a wall-crossing function we decorate them with an integral tangent vector.

\begin{definition}[Wall]
    A \emph{wall} in $\Sigma$ is a pair $(\fd , v)$ where $\fd\subset \Sigma$ is a closed rational polyhedral cone of codimension $1$, contained in a cone of $\Sigma^{\ess}$, and $v$ lies in the integral tangent space to $\fd$ for the affine structure on $\Sigma_0$ (see \eqref{eq:remove-codim-2-locus}).
    The \emph{support of the wall} is the cone $\fd$.
\end{definition}

To a fixed set of walls, we associate a balancing condition.

\begin{definition}[Balancing condition]
    Let $W$ be a set of walls in $\Sigma$.
    Let $h\colon \Gamma_{\circ}\rightarrow \Sk (U)\simeq \vert \Sigma^{\ess}\vert$ be a $\Sigma$-piecewise linear map, where $\Gamma_{\circ}$ is a metric tree. 
    \begin{enumerate}[wide]
        \item Let $v\in h^{-1} (\vert \Sigma_0\vert)$ be a vertex.
        The \emph{bend of $h$ at $v$} is the integral tangent vector $\NB_v\coloneqq\sum_{e\ni v} d_eh$, where edges are oriented towards $v$.
        Here we use the integral affine structure on $\Sigma_0$.
        \item We say that $h$ satisfies the \emph{balancing condition with respect to $W$} if for each vertex $v\in h^{-1} (\vert \Sigma_0\vert)$ the bend $\NB_v$ is either $0$, or there exists $(\fd, w)\in W$ such that $h(v)\in \fd$ and $\NB_v =w$.
    \end{enumerate}
\end{definition}

\subsubsection{Non-archimedean walls and transverse spines}

The next proposition follows from \cite[Proposition 3.13]{Keel-Yu_Log-CY-mirror-symmetry}, after identifying $\Sk (U)$ with $\vert \Sigma^{\ess}\vert$.

\begin{proposition}[Non-archimedean walls] \label{constr:analytic-walls}
    Fix a curve class $A\in \NE (X,\bbZ )$.
    There exists a finite set of walls $\Wall_A^{\an}$ in $\Sigma$ such that for any $n\geq 1$ and any analytic stable map $f\colon (C,(p_1,\dots, p_n))\rightarrow X_{\eta}^{\an}$ with $f_{\eta}^{-1} (D_{\eta}^{\an} )$ supported on $\lbrace p_i\rbrace_{1\leq i\leq n}$ and $A-f_{\ast} [C]$ effective, the spine $\Sp (f)$ is satisfies the balancing condition with respect to $\Wall_A^{\an}$.
\end{proposition}

The fact that analytic spines satisfy a balancing condition motivates the following general definition.
It is convenient to have a general notion of spine independent of realizability by an analytic stable map.

\begin{definition}[Abstract spine]
    Let $W$ be a finite set of walls in $\Sigma$.
    Let $h\colon \Gamma_{\circ}\rightarrow \vert \Sigma^{\ess}\vert$ be a continuous map, with $\Gamma_{\circ}$ a metric graph of genus $0$, equal to the convex hull of its infinite legs.
    We say that $h$ is a \emph{$W$-spine} if
    \begin{enumerate}[wide]
        \item $h$ is $\Sigma$-piecewise linear, 
        \item the tropical type associated to $h$ is simple (see \cref{def:spine-type}),
        \item $h$ satisfies the balancing condition with respect to $W$, and
        \item the derivative of $h$ on each leg is an integral tangent vector in $\Sigma^{\ess}$ (for the piecewise linear affine structure).
    \end{enumerate}
\end{definition}

To define non-archimedean counts, the notion of transverse spine is introduced.

\begin{definition}[Transverse spine]
    Let $W$ be a finite set of walls in $\Sigma$.
    A $W$-spine $h\colon \Gamma_{\circ}\rightarrow \vert \Sigma^{\ess}\vert$ is called \emph{$W$-transverse} if:
    \begin{enumerate}[wide]
        \item $h(\Gamma_{\circ} )$ is transverse to $W$, i.e. $h(\Gamma )$ intersects supports of walls at finitely many points and contains no points in $(d-2)$-dimensional strata,
        \item every vertex of $\Gamma_{\circ}$ whose image lies in the support of a wall is $2$-valent, 
        \item $h(\Gamma )$ does not intersect $\vert \Sigma^{\ess}\vert\setminus \vert \Sigma_0\vert$.
    \end{enumerate}
\end{definition}

For $W$ a finite set of walls in $\Sigma$ and a finite set $J$, we denote by $\SP_J (\Sigma , W)$ the set of $W$-spines with leaves indexed by $J$ and by $\SP_J^{\tr} (\Sigma, W)\subset \SP_J (\Sigma ,W)$ the subset of $W$-transverse spines.
Both spaces are equipped with a natural Hausdorff topology (see \cite[Definition 4.19]{Keel_Yu_The_Frobenius}).
We also denote by $\overline{M}_{0,J}^{\trop}$ the space of extended nodal metric trees with leaves indexed by $J$, and with no finite $2$-valent vertices.
Transverse spines satisfy the following rigidity property.

\begin{proposition}[{\cite[Proposition 3.25]{Keel-Yu_Log-CY-mirror-symmetry}}] \label{prop:rigidity-transverse-spines}
    Let $W$ be a finite set of walls in $\Sigma$ and $J$ a finite set.
    Let $\SP_J^{tr} (\Sigma ,W)$ be the set of $W$-transverse spines with leaves indexed by $J$, and let $u$ be a leaf of $\Gamma_{\circ}$ with associated contact order $0$.
    Then the map 
    \[(\dom ,\ev_u)\colon \SP_J^{\tr} (\Sigma ,W)\longrightarrow \overline{M}_{0,J}^{\trop} \times \vert \Sigma^{\ess}\vert,\]
    taking the simplification of the domain and evaluating at the leaf $u$ is a local homeomorphism.
\end{proposition}

\personal{
\begin{proposition}[{\cite[Proposition 3.13]{Keel-Yu_Log-CY-mirror-symmetry}}]
    Fix a curve class $A\in\NE (X,\bbZ )$.
    There exists a finite collection of walls in $\Sigma$ such that for any $n\in\bbN$ and any analytic stable map $f\colon (C, (p_1,\dots, p_n))\rightarrow X_{\eta}^{\an}$ satisfying $f^{-1} (D_{\eta}^{\an})$ is supported on the marked points and $A-f_{\ast}[C]$ is effective, the following holds:
    \begin{enumerate}[wide]
        \item Let $h\colon T\rightarrow \Sk (U)\simeq \vert \Sigma^{\ess}\vert$ be a twig associated to $f$.
        For each pair $(v,e)$ of vertex and incident edge in $v$, there exists a wall $(\fd, w)$ such that $h(v)\in\fd$ and $w = d_e h$, where $e$ is oriented away from the root.
        \item Let $\Sp (f)\colon \Gamma_{\circ}^s\rightarrow \Sk (U)$ be the spine of $f$.
        For any $x\in \Gamma_{\circ}^s$, either $\Sp (f)$ is balanced at $x$, or there is a wall $(\fd ,v)$ with $\Sp(f)(x)\in\fd$ and either $\Sp (f)(x)\in \Sing (\Sk (U))\simeq \Sigma^{\ess}\setminus \Sigma_0$, or the bend of $\Sp (f)$ at $x$ is equal to $v$.
    \end{enumerate}
\end{proposition}}

\subsubsection{Logarithmic walls and broken lines}
%\subsubsection{Wall types and broken lines}

In the logarithmic setting, the canonical wall structure on $\Sigma\simeq \Sigma (X_s,D_s)$ is defined using wall types (see \cite{Gross_Canonical-wall-structure-and-intrinsic-mirror-symmetry}).
We define a collection of walls in $\Sigma$ with the same support, and using contact orders of wall types as decorations.

\begin{definition}[Wall type]
    A \emph{wall type} is a type $\tau = (G,\bsigma, \bu)$ of tropical map to $\Sigma (X_s,D_s)$ such that:
    \begin{enumerate}[wide]
        \item $G$ is a genus zero graph with $L(G) = \lbrace L_{\out}\rbrace$ with $\bsigma (L_{\out})\in \Sigma^{\ess}$ and $\bu (L_{\out} )\neq 0$,
        \item $\tau$ is realizable and balanced, and
        \item  $\dim \tau = \dim h(\tau_{\out} ) = d-1$, where $\tau_{\out}$ is the cone parametrizing the leg $L_{\out}$.
    \end{enumerate}
\end{definition}

For a decorated wall type $\btau$ the moduli space $\cM (X_s , \btau )$ is proper of virtual dimension $0$, producing the wall type invariant 
\[W_{\btau}\coloneqq \frac{\deg [\cM (X_s, \btau )]^{\vir}}{\vert \Aut (\btau )\vert} .\]
The restriction $h\vert_{\tau_{\out}}\colon \tau_{\out}\rightarrow \bsigma (L_{\out} )$ induces a morphism $h_{\ast}\colon \Lambda_{\tau_{\out}}\rightarrow \Lambda_{\bsigma (L_{\out} )}$ between tangent lattices, and we define 
\[k_{\tau} \coloneqq \vert \coker (h_{\ast} )_{\tors} \vert = \Bigg\vert \bigg(\frac{\Lambda_{h(\tau_{\out})}}{h_{\ast} (\Lambda_{\tau_{\out}} )}\bigg)_{\tors}\Bigg\vert .\]
Note that this is an integer which does not depend on the curve class decoration, justifying the notation. 
We have the following result, analogous to \cite[Proposition 3.8]{Keel-Yu_Log-CY-mirror-symmetry}.

\begin{lemma} \label{lemma:finitely-many-wall-types}
    Fix an effective curve class $A$.
    There are finitely many decorated wall types $\btau$ with total curve class $A'$ such that $A-A'$ is effective and $W_{\btau}\neq 0$.
\end{lemma}

\begin{proof}
    \todo{Double check proof.}

    Let $\btau = (\tau ,\bA')$ with $\tau  =(G_{\tau} ,\bsigma, \bu )$ be a decorated wall type with total curve class $A'$, such that $A-A'$ is effective and $W_{\btau}\neq 0$.
    In particular, $\btau$ is the type of a punctured log stable map $f\colon (C,x_{\out})\rightarrow X$.
    Using \cite[Corollary 1.14]{Gross_Intrinsic_mirror_symmetry} to relate contact orders to the curve class decoration, we produce a logarithmic version of the arguments of \cite[Proposition 3.8]{Keel-Yu_Log-CY-mirror-symmetry}.
    
    First, we prove that there are finitely many topological types for the domain by bounding the number of $1$-valent vertices.
    By strict convexity of the Mori cone, there are finitely many curve classes $A'\in \NE (X,\bbZ )$ such that $A-A'$ is effective. 
    Similarly, for any subset of vertices $V'\subset V(G_{\tau} )$ there are finitely many possiblities for $\sum_{v\in V'} \bA '(v)$ with a uniform bound depending on $A$ only. 
    A $1$-valent vertex $v$ corresponds to an irreducible component $C_v\subset C$ with a single node and at most one marked point.
    By stability, it is not contracted by $f$ and so $\bA' (v) = f_{\ast} [C_v]\neq 0$.
    In this way, we obtain a bound depending on $A$ on the number of $1$-valent vertices in $G_{\tau}$.
    It follows that there are finitely many topological types for the domain $G_{\tau}$.

    Next, we prove that there are finitely many possibilities for the contact orders on edges adjacent to $1$-valent vertices and on the leg $L_{\out}$.
    We have
    \begin{align*}
        A'\cdot D_i &= \langle \bu (L_{\out}) , D_i\rangle, \\
        \bA' (v)\cdot D_i &= \langle \bu (e) , D_i\rangle
    \end{align*}
    for each irreducible component $D_i$ of $D$ and pair $(v,e)$ of $1$-valent vertex and edge containing $v$ oriented away from $v$.
    Hence, the contact orders $\bu (L_{\out} )$ and $\bu (e)$ are uniquely determined by $A'$ and $\bA' (v)$.
    Since there are finitely many choices for those curve classes, there are also finitely many choices for the contact orders.

    Finally, we prove that there are finitely many possibilities for the function $\bu$. 
    For a vertex $v\in V(G_{\tau} )$, the total bend $\sum_{e\ni v} \bu (e)$ at $v$ (where edges are oriented away from $v$) is uniquely determined by the intersection numbers of the curve class decoration $\bA '(v)$ with the irreducible components of $D$.
    It follows that for any subset of vertices $V'\subset V(G_{\tau} )$, there are finitely many possibilities for $\sum_{v\in V'}\sum_{e\ni v} \bu (e)$.
    We conclude as in \cite{Keel-Yu_Log-CY-mirror-symmetry} that there are finitely many possibilities for $\bu$.
    The proof is complete.
\end{proof}

The previous lemma determines a finite set of walls $\Wall_A^{\log}$ for each effective curve class $A$, which we describe in the following construction. 

\begin{construction}[Logarithmic walls] \label{constr:logarithmic-walls}
    Fix an effective curve class $A\in \NE (X, \bbZ )$.
    To each decorated wall type $\btau$ as in \cref{lemma:finitely-many-wall-types} we associate a wall $(\fd_{\tau} , -u_{\tau} )\coloneqq (h(\tau_{\out} ) , -\bu (L_{\out} ))$.
    We then define
    \begin{align*}
        \Wall_A^{\log} \coloneqq \bigcup_{k\geq 1} \bigcup_{\substack{\tau_1,\dots, \tau_k \\ \dim \fd_{\tau_1}\cap\cdots\cap \fd_{\tau_k} = d-1}} \lbrace (\fd_{\tau_1}\cap \cdots \cap \fd_{\tau_k} ,-( u_{\tau_1} + \cdots + u_{\tau_k} ))\rbrace ,
    \end{align*}
    where the second union is over (undecorated) wall types which arise from \cref{lemma:finitely-many-wall-types}.  
    In particular, there are finitely many walls in $\Wall_A^{\log}$.
\end{construction}

We recall the notion of broken line type, which we will use when describing tropical cylinders.

\begin{definition}[Broken line type]
    A \emph{(non-trivial) broken line type} is a type $\omega = (G,\bsigma, \bu )$ of tropical map to $\Sigma (X_s,D_s)$ such that: 
    \begin{enumerate}[wide]
        \item $G$ is a genus zero graph with $L(G) = \lbrace L_{\inc} , L_{\out} \rbrace$ with $\bsigma (L_{\out} )\in \Sigma^{\ess}$ and 
        \[\bu (L_{\out} )\neq 0 ,\quad \bu (L_{\inc} )\in \bsigma (L_{\inc} )\setminus \lbrace 0\rbrace ,\]
        \item $\omega$ is realizable and balanced, and
        \item $\dim \omega = d-1$ and $\dim h(\omega_{\out} ) = d$, where $\omega_{\out}$ is the cone parametrizing the leg $L_{\out}$.
    \end{enumerate}
    \todo{Check if we need trivial broken line types in the paper (Single leg and no vertices), or degenerate ones ($\dim \omega = n-2$).}
\end{definition}

For a (non-trivial) decorated broken line type $\bomega$ the moduli space $\cM (X_s ,\bomega )$ is proper of virtual dimension $0$, producing the broken line invariant 
\[N_{\bomega} \coloneqq \frac{\deg [\cM (X_s ,\bomega )]^{\virt}}{\vert \Aut (\bomega )\vert} .\]
Similar to wall types, the restriction $h\vert_{\omega_{\out}}\colon \omega_{\out}\rightarrow \bsigma (L_{\out} )$ induces a morphism $h_{\ast}\colon \Lambda_{\omega_{\out}}\rightarrow \Lambda_{\bsigma (L_{\out} )}$ between tangent lattices, this time of finite index, and we define 
\[k_{\omega} \coloneqq \vert \coker (h_{\ast} ) \vert = \Big\vert\frac{\Lambda_{\bsigma (L_{\out})}}{h_{\ast} (\Lambda_{\omega_{\out}} )}\Big\vert .\]

The next lemma establishes a balancing condition analogous to the one satisfied by spines of analytic stable maps.
\personal{Balancing condition: for each vertex in the spine whose image cone is not contained in the singular locus, either $v$ is balanced, or it is contained in a wall.}
\begin{lemma} \label{lemma:balancing-condition-broken-line}
    Let $\bomega$ be a decorated broken line type of total curve class $A$, and assume $N_{\bomega}\neq 0$.    
    Denote by $G_{\omega}'\subset G_{\omega}$ the convex hull of the legs.
    Each vertex $v$ in $G_{\omega}'$ satisfies one of the following:
    \begin{enumerate}[wide,label = (\roman*)]
        \item $v$ is $2$-valent (in $G_{\omega}$) and balanced, or
        \item there exists $(\fd , u)\in \Wall_A^{\log}$ such that $h(\omega_v)\subset \fd$ and $\sum_{e\ni v} \bu (e) = u$, where the sum is over edges in $G_{\omega}'$ oriented away from $v$.
    \end{enumerate}
    \personal{balancing condition makes sense because $\dim\bsigma (v)\geq n-1$.}
\end{lemma}

\begin{proof}
    By \cite[Construction 4.12]{Gross_Canonical-wall-structure-and-intrinsic-mirror-symmetry}, splitting $\bomega$ at the edges adjacent to vertices in $G_{\omega}'$ which are not in $G_{\omega}'$ produces decorated wall types.
    If the broken line invariant $N_{\bomega}$ is not zero, then the invariants associated to those wall types are all nonzero by \cite[Eq.(4.2)]{Gross_Canonical-wall-structure-and-intrinsic-mirror-symmetry}.
    In particular, the total curve classes of those decorated wall types are effective.
    It follows that for any such wall type $\btau$ of total curve class $A'$, the curve class $A-A'$ is effective.
    Hence, each of the wall types produced by $\bomega$ satisfies the conditions in \cref{lemma:finitely-many-wall-types}.

    For every vertex $v$ in $G_{\omega'}$ we have $\dim h_{\omega} (\omega_v) = d-1$ (see \cite[Lemma 2.5]{Gross_Canonical-wall-structure-and-intrinsic-mirror-symmetry}), and hence $\dim\bsigma (v)\geq d-1$.
    In particular, the balancing condition \cite[Eq. (2.4)]{Gross_Tropical_geometry_and_mirror_symmetry} holds at every vertex of $G_{\omega}'$.
    If $v$ has valency $2$ in $G_{\omega}$, then its adjacent edges (or legs) are all in $G_{\omega}'$.
    Thus the restriction of $h_{\omega}$ to the spine is still balanced at $v$, and $v$ satisfies (i).   
    
    Now, let $v$ be a vertex in $G_{\omega}'$ which is not $2$-valent (in $G_{\omega}$).
    Let $\btau_1,\dots ,\btau_k$ denote the decorated wall types glued to $v$.
    For each $1\leq i\leq k$, the type $\btau_i$ satifies the conditions of \cref{lemma:finitely-many-wall-types} and $h_{\omega} (\omega_v) \subset \fd_{\tau_i}$.
    Let $e,e'$ be the edges in $G_{\omega'}$ adjacent to $v$, oriented away from $v$.
    By the balancing condition, we have $\bu (e)+\bu (e') = u_{\tau_1} + \cdots + u_{\tau_k}$.
    Note that $\dim \fd_{\tau_1}\cap \cdots \cap \fd_{\tau_k} = d-1$ since the intersection contains $h_{\omega} (\omega_v)$.
    Thus, the pair $(\fd_{\tau_1}\cap \cdots \cap \fd_{\tau_k} , -(u_{\tau_1} + \cdots + u_{\tau_k} ))$ is a wall in $\Wall_A^{\log}$ that satisfies (ii).
    The lemma is proved.
\end{proof}

The next result shows that generically, the spine of a broken line type is transverse to a given set of walls. 

 \begin{lemma} \label{lemma:transversality-broken-lines}
   Let $\omega = (G,\bsigma ,\bu)$ be a broken line type, let $h\colon \Gamma (G,\ell ) \rightarrow \Sigma$ denote the associated universal family of tropical maps to $\Sigma$, and let $G'\subset G$ denote the convex hull of $L(G)$.
   Let $W$ be a finite set of supports of walls.
   There exists an open subset $\omega^{\circ}\subset \omega$ such that for all $t\in \omega^{\circ}$:
   \begin{enumerate}[wide]
     \item $h_t (G')$ only intersects codimension $0$ and $1$ cones of $\Sigma^{\ess}$, 
     \item the intersection $h_t(G')\cap W$ is finite, and
     \item if $h_t (G')$ intersects a $(d-2)$-dimensional stratum of $W$, this intersection is a single point given by the non-vertex endpoint of $L_{\out}$.
   \end{enumerate}
   Furthermore, $\omega^{\circ}$ can be obtained as the complement of a polyhedral subset of dimension at most $d-2$.
 \end{lemma}
 
 \begin{proof}
    For an edge $e$ in $G$, we denote by $u_e\coloneqq \bu (e)$ the slope of $h$ along $e$.
   By \cite[Lemma 2.5(2)]{Gross_Canonical-wall-structure-and-intrinsic-mirror-symmetry}, the broken line type $\omega$ satisfies:
   \begin{enumerate}[wide]
     \item for $t\in \Int (\omega )$, the image $h_t (G')$ only intersects codimensions $0$ or $1$ cones of $\Sigma^{\ess}$, except possibly for the non-vertex endpoints of the legs,
     \item for every vertex $v$ in $G'$, we have $\dim h(\omega_v) = d-1$, and
     \item for every edge $e$ adjacent to a vertex in the spine, either $e$ is in $G'$ and $u_e$ is not tangent to $h(\omega_v)$, or $e$ is not in $G'$ and $u_e$ is tangent to $h(\omega_v)$.
   \end{enumerate}
   Note that the second condition implies that for each vertex $v$ of $G'$, the evaluation map $h\vert_{\omega_v}\colon \omega_v\rightarrow \bsigma (v)$ is an isomorphism onto its image. The first condition is the property (1) of the lemma.
   
   We construct a lower dimensional polyhedral subset of $\omega$ such Conditions (2-3) of the lemma are satisfied when restricting $h$ to its complement.
   Fix two supports of walls $\fd$ and $\fd'$ in $W$, with $\dim \fd\cap\fd '\leq d-2$.
   Let $e$ be an edge in the spine, and $v$ be a vertex adjacent to $e$.
   If there exists $t\in \omega$ such that $h_t(e)\subset \fd$, then necessarily $\fd\cap h(\omega_v)\neq \emptyset$ and $u_e$ is tangent to $\fd$.
   Since $u_e$ is not tangent to $h(\omega_v)$, the dimension of the intersection $\fd\cap h(\omega_v)$ is at most $d-2$.
   In that case the locus $\lbrace t\in \omega \mid h_t (e)\subset \fd\rbrace$ coincides with $(h\vert_{\omega_v})^{-1} (\fd)$, which is a polyhedral subset of dimension at most $d-2$ since the evaluation map $h\vert_{\omega_v}\colon \omega_v\rightarrow \bsigma (v)$ is injective.
   For $t$ in the complement of that subset the image $h_t(e)$ does not intersect $\fd$.

   For the same reason, the locus $(h\vert_{\omega_v})^{-1} (\fd\cap\fd' )$ is a polyhedral subset of dimension at most $d-2$.
   Assume now that there exists $t\in \omega$ such that $h_t(e)$ intersects $\fd\cap \fd'$ at a single point, in particular $u_e$ is not tangent to either $\fd$ or $\fd'$.
   Consider the evaluation map $h\vert_{\omega_e}\colon \omega_e\rightarrow \sigma (e)$ given by $(t,\lambda)\mapsto h_t(v) + \lambda u_e$.
   Note that $\dim h(\omega_e) = d$ since $u_e$ is not tangent to $h(\omega_v)$, so this map is injective.
   Then $(h\vert_{\omega_e})^{-1} (\fd\cap\fd' )$ is a polyhedral subset of dimension at most $d-2$, and the same is true of its image in $\omega$ under the projection $\omega_e\rightarrow \omega$.
   Taking the union of all the constructed subsets produces a polyhedral subset in $\omega$ of dimension at most $d-2$, whose complement satisfies the desired properties.
   The proof is complete.

   \personal{Let $W$ be the set of walls.
   Let $G_{\omega}'\subset G_{\omega}$ denote the domain of the spine, $\pi\colon \Gamma (G_{\omega},\ell )\rightarrow \omega$ the projection.
   The subset constructed is:
   \[\bigcup_{v\in V(G_{\omega}')}\bigcup_{\substack{\fd\in W\\ \dim \fd\cap h(\omega_v)\leq d-2}} (h\vert_{\omega_v})^{-1} (\fd)
   \cup \bigcup_{e\in E(G_{\omega}')} \bigcup_{\substack{\fd,\fd'\in W\\ u_e\notin T\fd, u_e\notin T\fd'}} \pi \big ( (h\vert_{\omega_e})^{-1} (\fd\cap\fd')\big) .\]
   }
 \end{proof}

\subsection{Cylinder counts} \label{subsec:cylinder-counts}

Non-archimedean and logarithmic cylinder counts are a special type of $3$-pointed invariants with some combinatorial constraint encoded by the analytic spine, or the spine type associated to a log stable map.
In this subsection, we define those invariants and establish some properties of tropical cylinder types. 
We also prove the birational invariance of logarithmic cylinder counts under log \'{e}tale modifications in \cref{prop:birational-invariance-log-cylinders}.

Given an effective curve class $A$, we produced two finite sets of walls in $\Sigma$, the non-archimedean walls $\Wall_A^{\an}$ (\cref{constr:analytic-walls}) and the logarithmic walls $\Wall_A^{\log}$ (\cref{constr:logarithmic-walls}).
It is not clear that those two sets of walls agree, so we consider their union.
In the surface case, \cref{thm:exponential-formula} implies that $\Wall_A^{\an}$ and $\Wall_A^{\log}$ agree. 

\begin{construction}[Total set of walls] \label{constr:total-walls}
    Fix an effective curve class $A\in \NE (X,\bbZ )$.
    We define the total set of walls associated to $A$ as
    \begin{align}
        \Wall_A \coloneqq \Wall_A^{\an}\cup \Wall_A^{\log} .
    \end{align}
    By construction, $\Wall_A$ contains finitely many walls.     
\end{construction}

\subsubsection{Non-archimedean cylinder counts}
Non-archimedean cylinder counts enumerate $3$-pointed analytic stable maps with a fixed spine, called a cylinder spine.
We fix the data $\beta = (A, \bu )$ consisting of an effective curve class $A\in \NE (X,\bbZ)$ and $\bu = (\bu_1,\bu_2, \bu_i)$ a triple of contact orders with $\bu_i = 0$.

We consider a toric blowup 
\[\pi\colon (\tX ,\tD)\longrightarrow (X,D)\]
such that $\bu_1$ and $\bu_2$ correspond to rays of the dual fan of $(\tX,\tD)$, this does not change the interior $U\subset X$.
Tropically, this corresponds to taking a subdivision $\tSigma$ of the dual fan $\Sigma$. 
In particular, the essential skeleton $\vert \Sigma^{\ess}\vert$ does not change.

To avoid confusions, when necessary we write $\cM (U_{\eta}\subset X_{\eta} , \beta )$ to specify that we are looking at a moduli space of stable maps in $(X,D)$, and similarly write $\cM^{\sm} (U_{\eta}\subset X_{\eta} , \beta )$ for the smooth locus. 
The following proposition relates the moduli spaces before and after the blowup.

\begin{proposition}[{\cite[Proposition 7.12]{Keel-Yu_Log-CY-mirror-symmetry}}] \label{prop:moduli-spaces-before-after-blowup}
    The moduli space $\cM^{\sm} (U_{\eta}\subset X_{\eta} ,\beta )$ is not empty if and only if there exists $\tA\in \NE (\tX,\bbZ)$ such that $\cM^{\sm} (U_{\eta} \subset \tX_{\eta} ,(\tA,\bu ))$ is not empty. 
    In this case, $\tA$ is the unique $\gamma\in\NE (\tX,\bbZ )$ such that $\pi_{\ast} \gamma = A$ and $\cM^{\sm} (U_{\eta}\subset \tX_{\eta} , (\gamma,\bu ))\neq \emptyset$.
\end{proposition}

In the context of the above proposition, we write $\tbeta = (\tA , \bu )$ for the unique curve class and contact orders to $\tX$ induced by $\beta$.
For future reference, we note that the class $\tA$ is characterized by $\pi_{\ast}\tA = A$ and its intersection numbers with the exceptional divisors of $\pi$, which are determined by $(\bu_1,\bu_2)$.
If $E$ is an exceptional component corresponding to a ray $\rho$ in $\tSigma$, the intersection number $\tA\cdot E$ is the integral length of $\bu_1$ (resp. $\bu_2$) if $\rho$ and $\bu_1$ (resp. $\bu_2$) have the same direction, and $0$ otherwise.

We now consider the evaluation map to the third marked point
\[\ev_i\colon \cM (U_{\eta},\tbeta) \longrightarrow \tX,\]
and the map taking the spine of an analytic stable map
\[\Sp\colon \cM (U_{\eta}^{\an},\tbeta ) \longrightarrow \NT ( \tSigma, U),\]
where $\NT (\tSigma, U)$ is the set of $\tSigma$-piecewise linear maps from an extended nodal metric tree to $\Skbar^{\ess} (\tX,\tD)$, such that the preimage of the boundary $\Skbar^{\ess} (\tX,\tD)\setminus \Sk (\tX,\tD)$ is supported on the first and second marked points.

Non-archimedean cylinder counts are obtained by computing a degree after fixing a point constraint and a transverse spine.
We give the precise definition of a cylinder spine below, a typical cylinder spine is depicted in \cref{fig:cylinder-spine}.

\begin{definition}[Cylinder spine] \label{def:cylinder-spine}
    A \emph{cylinder spine of type $\beta$} is a $\Wall_A$-spine $h\colon \Gamma_{\circ}\rightarrow \vert \Sigma^{\ess}\vert$ such that:
    \begin{enumerate}[wide]   
        \item $h$ is a $\Wall_A$-transverse spine,
        \item the contact orders along the legs of $\Gamma_{\circ}$ are given by $\bu$, and
        \item the contracted leg of $h$ is mapped to the interior of a maximal cone of $\Sigma^{\ess}$.     
    \end{enumerate}
    \personal{Do we need: $h$ arises as the spine of an analytic stable map of total curve class $A$?}
\end{definition}

\begin{figure}[h!]
    \centering
    \setlength{\unitlength}{0.5\textwidth}
    \begin{picture} (1.5,1)
    %\includesvg[width=0.7\linewidth]{Cylinder-spine.svg}
    \put(0,0){\includegraphics[height=0.95\unitlength]{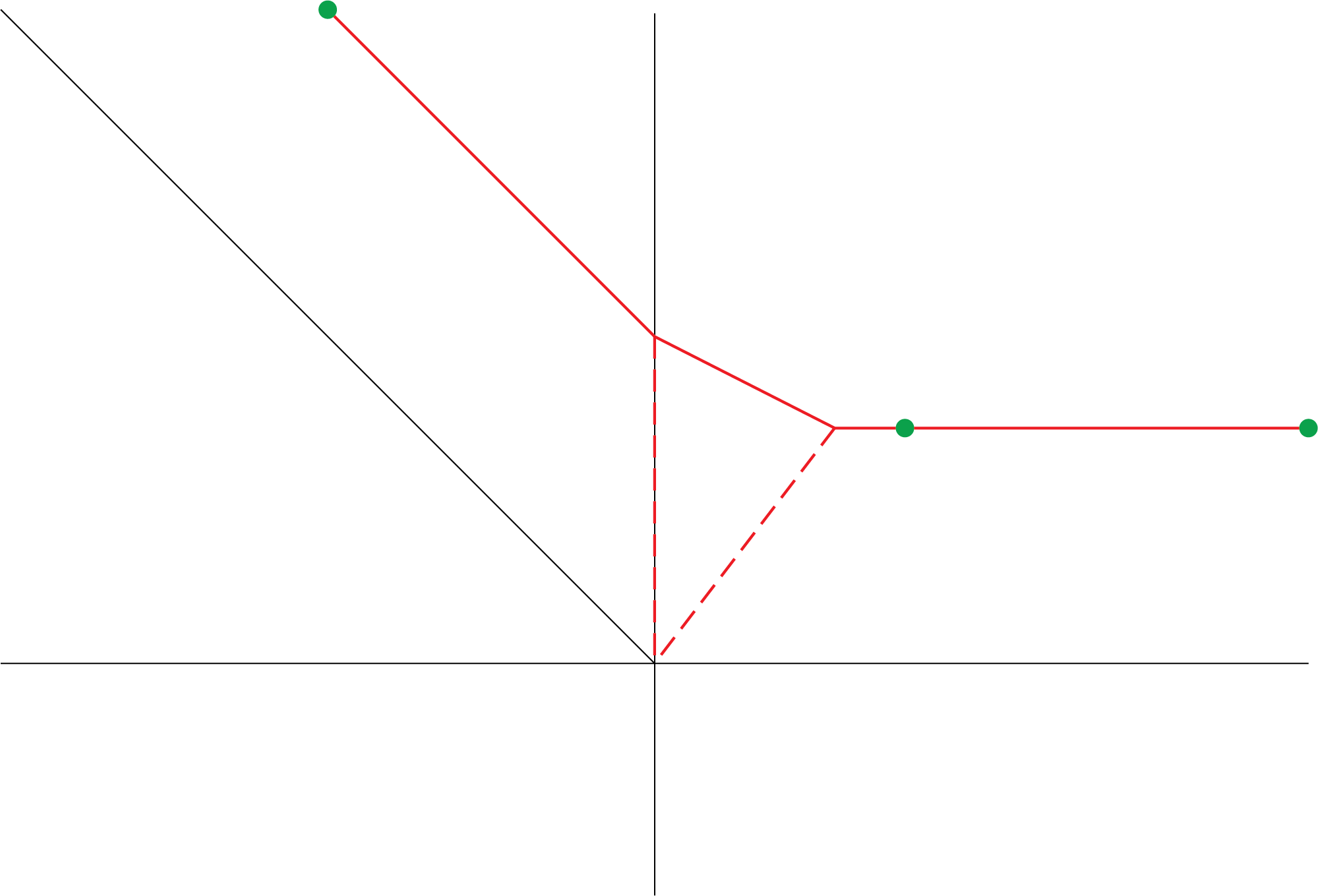}}
    \put(0.33,0.96){$v_1$}
    \put(1.36,0.52){$v_2$}
    \put(0.945,0.52){$v_i$}
    \end{picture}
    \caption{A typical cylinder spine (in solid red). A tropical cylinder type associated to this spine has wall types attached to the bending vertices (in dashed red).}
    \label{fig:cylinder-spine}
\end{figure}
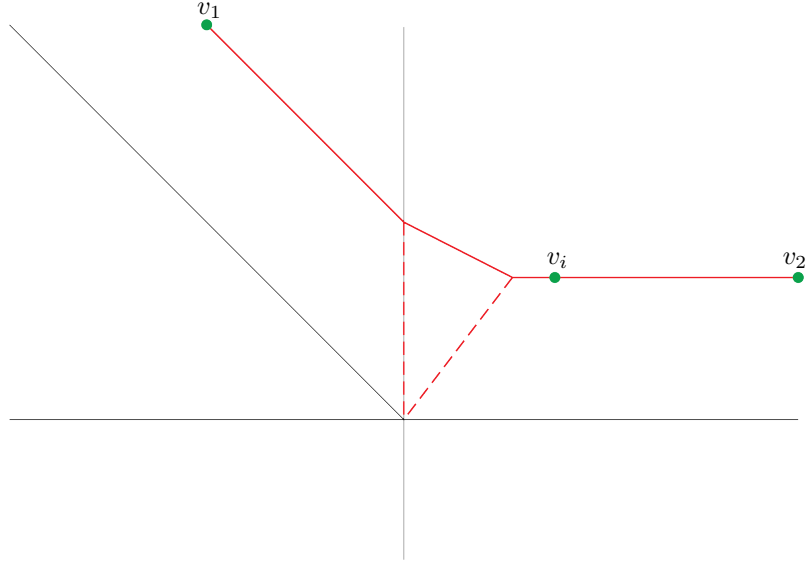

A cylinder spine $S$ of type $\beta$ is a $\Sigma$-piecewise linear map.
It induces a unique $\tSigma$-piecewise linear map obtained by subdividing the domain, which we denote by $\tS$.
The transversality condition ensures that $\Sp^{-1} (\tS)$ is contained $\cM^{\sm} (U_{\eta}^{\an} ,\tbeta )$, in particular the restriction of $\ev_i$ is \'{e}tale.
If $\cM^{\sm} (U_{\eta},\beta)\neq\emptyset $, we define the \emph{non-archimedean cylinder count} associated to $S$ by:
\begin{equation}\label{eq:def-analytic-cylinder-counts}
    N(S, A)\coloneqq \deg \big( \ev_i\vert_{\Sp^{-1} (\tS)}\big) .
\end{equation}
We define the count to be $0$ if the moduli space is empty.
It follows from \cite[Theorem  8.11]{Keel-Yu_Log-CY-mirror-symmetry} that this count is invariant when we deform the spine $S$ in the space of $\Wall_A$-transverse spines.
In order to make connections with logarithmic invariants, we will use this deformation invariance to guarantee that the lift $\tS$ is still a cylinder spine without changing the count.

\begin{lemma} \label{lemma:transverse-spine-after-blowup}
    Assume $\cM^{\sm} (U_{\eta} , \beta )\neq \emptyset$, let $S$ be a cylinder spine of type $\beta$.
    There exists a small deformation $S'$ of $S$ in $\NT(\Sigma, U)$ such that $\tS'$ is a cylinder spine of type $\tbeta$.
\end{lemma}

\begin{proof}
    \todo{Double check proof.}
    Let $\Wall_{\tA}$ denote the wall structure in $\tSigma$ associated to the curve class $\tA$.
    Since $\vert \Sigma\vert = \vert \tSigma \vert$, this defines a wall structure in $\Sigma$ as well. 
    Let $W$ be the union of the wall structures $\Wall_A$ and $\Wall_{\tA}$, together with codimension $1$ cones in $\tSigma$.
    The rigidity property for transverse spines (\cref{prop:rigidity-transverse-spines}) implies that $S$ can be deformed into a top-dimensional family of $\Wall_A$-transverse spines.
    Similar to the case of broken lines, we can show that in this family the image cone of every edge is top dimensional. 
    In particular, elements of this family will generically be $W$-transverse spines. 
    Then, their lifts to $\tSigma$-piecewise linear maps are cylinder spines of type $\tbeta$.
    The proof is complete. 
\end{proof}

\subsubsection{Logarithmic cylinder counts} \label{subsubsec:log-cylinder-counts}

\begin{definition}[Tropical cylinder type]\label{def:S-type}
    A \emph{tropical cylinder type} is a tropical type $\tau = (G,\bsigma,\bu)$ to $\Sigma (X_s,D_s)$ such that:
    \begin{enumerate}[wide]
        \item $G$ is a genus $0$ graph with $L(G) = \lbrace L_1,L_2,L_i\rbrace$ with $\bsigma (L_1),\bsigma (L_2)\in \Sigma^{\ess}$ and 
        \[\bu (L_1)\in \bsigma (L_1)\setminus \lbrace 0\rbrace,\;\bu (L_2)\in\bsigma (L_2)\setminus \lbrace 0\rbrace,\; \bu (L_i) =0,\]
        \item $\tau$ is realizable and balanced, and
        \item $\dim \tau = \dim h(\tau_{v_i}) = d$, where $v_i$ is the vertex adjacent to $L_i$.
    \end{enumerate}
    A \emph{decorated tropical cylinder type} is a decorated tropical type $\btau = (\tau , \bA )$, with $\tau$ a tropical cylinder type.
    A typical tropical cylinder type is depicted in \cref{fig:cylinder-spine}.
\end{definition}

For a decorated tropical cylinder type $\btau$, we define the associated logarithmic cylinder count as the logarithmic Gromov-Witten invariant: 
\[N_{\btau} \coloneqq \frac{\deg [\cM (X_s , \btau)]^{\vir}}{\vert \Aut (\btau )\vert}.\]
We have an induced map $h_{\ast}\colon \Lambda_{\tau_{v_i}}\rightarrow \Lambda_{\bsigma (v_i)}$ which has finite index,
and define 
\begin{equation}\label{eq:multiplicity-tropical-cylinder}
    k_{\tau}\coloneqq \vert\coker (h_{\ast} )\vert = \Big\vert \frac{\Lambda_{\bsigma (v_i)}}{h_{\ast} (\Lambda_{\tau_{v_i}})} \Big\vert.
\end{equation}

\personal{For a decorated tropical cylinder type $\btau$, we denote by $\beta = (\bu , A)$ the data consisting of the contact orders $\bu$ and the total curve class $A$.}

\begin{lemma} \label{lemma:splitting-cylinder-type}
    Let $\btau$ be a decorated tropical cylinder type as in \cref{def:S-type}.
    Then:
    \begin{enumerate}[wide]
        \item The vertex $v_i$ is exactly $3$-valent.
        Splitting $\btau$ at $v_i$ produces two decorated broken line types $\bomega_1$ and $\bomega_2$, and a decorated type $\btau_0$ with the single vertex $v_i$, three legs, no edges, $\dim \tau_0 = \dim h(\tau_{0,v_i} ) = d$, and curve class $0$.
        \item The spine type $\Sp (\tau )$ is induced by a tropical type in $\Sigma^{\ess}$.
        \item If $N_{\btau}\neq 0$, then
        \begin{align} \label{eq:splitting-logarithmic-cylinder-count}
            k_{\tau}N_{\btau} = k_{\omega_1} k_{\omega_2}N_{\bomega_1} N_{\bomega_2}.
        \end{align}
    \end{enumerate}
\end{lemma}

\begin{proof}
    Since $\bu (L_i )=0$, the balancing condition implies that $v_i$ is at least $3$-valent.
    Then (1) follows from \cite[Step 4]{Johnston_Comparison}.
    The spine type of a broken line type is induced by a tropical type in $\Sigma^{\ess}$ by \cite[Lemma 2.5]{Gross_Canonical-wall-structure-and-intrinsic-mirror-symmetry}, so the same holds for a cylinder type, proving (2).
    The formula in (3) is a consequence of the splitting formula, and a special case of \cite[Lemma 6.6]{Gross_Canonical-wall-structure-and-intrinsic-mirror-symmetry}.
    The proof is complete.
\end{proof}

\begin{proposition} \label{prop:transversality-cylinder-type}
    Let $\btau$ be a decorated tropical cylinder type with total curve class $A$ and $N_{\btau}\neq 0$.
    Then, there exists a $\Wall_A$-spine $S$ transverse to $\Wall_A$ such that $\Sp (\tau )$ has the same type as $S$.    
\end{proposition}

\begin{proof}
    By \cref{lemma:splitting-cylinder-type}, the type $\btau$ splits to produce two decorated broken line types $\bomega_1$ and $\bomega_2$, and a type $\btau_0$ with a single vertex and three legs.
    By \eqref{eq:splitting-logarithmic-cylinder-count}, we have $N_{\bomega_i}\neq 0$ for $i=1,2$.
    Furthermore $A = A_1+A_2$, where $A_i$ is total curve class of $\bomega_i$.
    In particular $A-A_i$ is effective, and we deduce from \cref{lemma:balancing-condition-broken-line} that the spine of $\bomega_i$ is generically a $\Wall_A$-spine.
    We deduce that the spine of $\btau$ is also a $\Wall_A$-spine.

    For $i=1,2$, we denote by $L_{\out,i}$ the leg of $\omega_i$ created by the splitting, and by $\omega_{\out,i}$ the cone parametrizing $L_{\out ,i}$. 
    We denote by $v_0$ the vertex of $\tau_0$ and by $\tau_{v,0}$ the cone parametrizing $v_0$ (which by definition equals $\tau_0$).
    The maps 
    \begin{align*}
        h_{\out,i}\colon \omega_{\out,i} &\longrightarrow h_{\omega_i} (\omega_{\out,i}), \\
        h_{\tau_{v,0}}\colon \tau_{v,0} &\longrightarrow h_{\tau_0} (\tau_{v,0} )
    \end{align*}
    are injective because their images have the same dimension as their domain.
    Thus any element $x\in h_{\out,1} (\omega_{\out,1})\cap h_{\out,2} (\omega_{\out,2} )\cap h_{\tau_{v,0}} (\tau_{v,0} )$ determines unique elements $t_i\in \omega_i$ for $i=1,2$ and $t_0\in\tau_0$.
    The maps $h_{\omega_i, t_i}$ and $h_{\tau_0,t_0}$ can be glued at $x$ to produce a tropical map which will be of the form $h_{\tau, t}$ for some $t\in \tau$ \personal{In fact $t = (h\vert_{\tau_{v_i}})^{-1} (x)$.}.
    Conversely, for each $t\in\tau$ the point $h_{\tau,t} (v_i)$ determines an element in $h_{\out,1} (\omega_{\out,1})\cap h_{\out,2} (\omega_{\out,2} )\cap h_{\tau_{v,0}} (\tau_{v,0} )$.
    Hence, the intersection contains $h_{\tau} (\tau_{v_i} )$ and in particular it is $d$-dimensional.

  We now prove that there exists an open locus $\tau^{\circ}\subset\tau$ such that the spine of $h_{\tau, t}$ is transverse to $\Wall_A$ for all $t\in\tau^{\circ}$.
  Applying \cref{lemma:transversality-broken-lines} with the set of walls $\Wall_A$, for $i=1,2$ we obtain $\omega_i^{\circ}\subset \omega_i$ defined as the complement of an $(d-2)$-dimensional polyhedral subset.
  Let $\omega_{\out,i}^{\circ}\subset\omega_{\out,i}$ denote the inverse image under the projection $\omega_{\out,i}\rightarrow \omega_i$, it is the complement of a $(d-1)$-dimensional polyhedral subset.
  Let $\tau_{v,0}^{\circ}\subset \tau_{v,0}$ denote the open locus where $v_0$ is not mapped into a wall, it is the complement of a $(d-1)$-dimensional polyhedral subset since $\tau_{v,0}\rightarrow h_{\tau_0} (\tau_{v,0})$ is injective.
  It follows that the intersection $h_{\out,1} (\omega_{\out,1}^{\circ})\cap h_{\out,2} (\omega_{\out,2}^{\circ} )\cap h_{\tau_{v,0}} (\tau_{0}^{\circ} )$ is the complement of a $(d-1)$-dimensional polyhedral subset in the $d$-dimensional polyhedral subset $h_{\out,1} (\omega_{\out,1})\cap h_{\out,2} (\omega_{\out,2} )\cap h_{\tau_{v,0}} (\tau_{v,0} )$.
  In particular, it is not empty and the set 
  \[\tau^{\circ}\coloneqq \ev_{v_i}^{-1} \big( h_{\tau_1} (\tau_{1,L_{\out,1}}^{\circ})\cap h_{\tau_2} (\tau_{2,L_{\out,2}}^{\circ})\cap h_{\tau_3} (\tau_{3,v}^{\circ} )\big) .\]
  is a non-empty open subset of $\tau$.
  For $t\in\tau^{\circ}$ the spine of $h_{\tau,t}$ is transverse to $\Wall_A$, as it is transverse when restricted to the subgraphs corresponding to $\omega_1$ and $\omega_2$, and the $3$-valent vertex $v_i$ is mapped away from walls by construction.
  The proposition is proved.
  
  \personal{
  By taking a point in $h_{\tau_1} (\tau_{1,L_1} )\cap h_{\tau_2,L_2} (\tau_{2,L_2})$ corresponding to $t_1\in\tau_1$ and $t_2\in\tau_2$, we can glue the maps $h_{\tau_1,t_1}$ and $h_{\tau_2,t_2}$ along that point.
  By adding a leg of contact order $0$ to the created vertex, we product a tropical map that has type $\tau$.
  So, it is enough to prove that the intersection $h_{\tau_1} (\tau_{1,L_1} )\cap h_{\tau_2,L_2} (\tau_{2,L_2})$ contains an open set.
  But this intersection contains the cone $h(\tau_{v_i})$, which is full dimensional.
  We conclude the proof of (1) using \cref{lemma:transversality-broken-lines}.}
\end{proof}

Combining \cref{prop:transversality-cylinder-type} with the rigidity property of transverse spines, we obtain the following rigidity property for spines of tropical cylinder types.
We recall that a tropical type $\tau'$ is \emph{marked by $\tau$} if it admits a contraction $\tau'\rightarrow \tau$.

 \begin{lemma}\label{lemma:degeneration-transverse-spine}
     Let $\btau$ be a decorated tropical cylinder type of total curve class $A\in\NE (X,\bbZ )$, and let $\tau'$ be a tropical type marked by $\tau$.
     If $N_{\btau}\neq 0$, then $\Sp (\tau ')$ has the same type as $\Sp (\tau )$.
 \end{lemma}

 \begin{proof}
    By \cref{prop:transversality-cylinder-type}, there exists a $\Wall_A$-transverse spine $S$ which has the type $\Sp (\tau )$.
    Such $S$ can be chosen to be the spine of $h_{\tau,t}$ for some $t\in \Int (\tau )$.
    In particular, $S$ maps the vertex $v_i$ to the interior of the image cone $h_{\tau} (\tau_{v_i})$.
    
    The type $\Sp (\tau' )$ is given by the type of a small deformation $S'$ of $S$.
    Since being $\Wall_A$-transverse is an open condition, we can assume that $S'$ is a transverse spine. 
    By \cref{prop:rigidity-transverse-spines}, such a small deformation of $S$ corresponds to a deformation the image of $S$ under the map $(\dom , \ev_i)$, remembering the simplification of the domain of $S$ as a metric tree and evaluating at the interior leg.
    
    Since $\dom (S)$ consists of a single vertex with three infinite legs, its only deformations turn infinite legs into finite legs.
    But $S'$ necessarily has infinite legs, so we may restrict to deformations of $S$ with fixed domain.
    Then $S'$ corresponds to a small deformation of $\ev_i (S)$, the image of the vertex $v_i$ under $S$.
    By definition of a cylinder spine the image cone $h_{\tau} (\tau_{v_i} )$ is $d$-dimensional, and we chose $S$ so that it maps $v_i$ to the interior of that cone.
    Then, any small deformation of $\ev_i (S)$ corresponds to the spine of an element of the universal family $h_{\tau}$.
    In particular, $S'$ has type $\Sp (\tau )$.
    The lemma is proved.
 \end{proof}

\subsubsection{Birational invariance of logarithmic cylinder counts}
Non-archimedean cylinder counts are defined after a toric blow up of $(X,D)$.
From the point of view of logarithmic geometry, toric blowups correspond to log-\'{e}tale modifications.
Here, we estabish the invariance of logarithmic cylinder counts under those modifications.

We fix a log-\'{e}tale modification $\pi\colon (\tX,\tD)\rightarrow (X,D)$ (i.e. a toric blowup), corresponding to a subdivision $\tSigma$ of $\Sigma$.
Given the data $\beta = (A,\bu )$ consisting of a curve class in $\NE (X ,\bbZ )$ and a tuple of contact orders to $\Sigma$, the paragraph after \cref{prop:moduli-spaces-before-after-blowup} describes a unique curve class $\tA$ in $\NE( \tX , \bbZ)$.
The following lemma provides a description of tropical lifts of decorated cylinder types. 
We refer to \cite{Johnston_Birational-invariance} for the notion of (decorated) tropical lift.

\begin{lemma}\label{lemma:tropical-lift-cylinder}
    Let $S$ be a cylinder spine of type $\beta = (A,\bu)$, such that the subdivided spine $\tS$ is a cylinder spine of type $\tbeta = (\tA,\bu)$.
    Fix a decorated tropical type $\btau$ to $\Sigma$, and a decorated lift $\btau'\rightarrow \btau$ to $\tSigma$.
    \begin{enumerate}[wide]
        \item Assume $\btau'$ is a decorated cylinder type in $\tSigma$ with spine type $\tS$ and $N_{\btau '}\neq 0$.
        Then $\btau$ is a decorated cylinder type in $\Sigma$ with spine type $S$.
        \item Assume $\btau$ is a decorated cylinder type in $\Sigma$ with spine type $S$.
        If $\dim\tau' = d$, then $\btau'$ is a decorated cylinder type in $\tSigma$ with spine type $\tS$.
    \end{enumerate}
\end{lemma}

\begin{proof}
\personal{We note that $\btau$ is uniquely determined by $\btau'$ (see \cite{Johnston_Birational-invariance}).\\}
Conditions (1) and (2) of \cref{def:S-type} are satisfied for both $\btau$ and $\btau'$ as soon as they are satisfied by either $\btau$ or $\btau'$.
We also have $\dim \tau \geq \dim \tau'$ because $\Int (\tau' )\subset \Int (\tau )$ by definition of a tropical lift.

To prove (1), assume that $\btau'$ is a decorated tropical cylinder type with spine type $\tS$ and $N_{\btau'}\neq 0$.
In particular $\dim \tau\geq \dim \tau' = d$.
If $\tau$ is not a tropical cylinder type, then necessarily $\dim \tau >d$ (note that $\dim h_{\tau} (\tau_{v_i} )\leq \dim \tau$).
We can then apply the same reasoning as \cite[Lemma 9.3]{Johnston_Birational-invariance} to conclude that $\deg [\cM (\tX ,\btau')]^{\vir} = 0$, contradicting the assumptions and proving that $\btau$ is a decorated cylinder type in $\Sigma$.
The spine type of $\btau$ is uniquely determined by $\Sp (\btau ')$ by removing redundant vertices, i.e. $2$-valent vertices which are mapped to a cone in $\tSigma\setminus \Sigma$.
This produces the spine type determined by $S$, proving (1).

For (2), assume that $\btau$ is a decorated tropical cylinder type with spine type $S$, and that $\dim\tau' = d$.
Let $v_i$ denote the vertex adjacent to the contracted leg in $\tau$.
Since the cone complex $\tau$ is identified with the image $h_{\tau} (\tau_{v_i} )$, the tropical lift $\tau'$ corresponds to a maximal cone in the subdivision of $h_{\tau} (\tau_{v_i} )$ induced by $\tSigma$.
In particular, $\btau'$ satisfies Condition (3) of \cref{def:S-type} and hence is a decorated tropical cylinder type.
The spine type $\Sp (\btau' )$ is obtained by subdividing $\Sp (\btau )$ according to $\tSigma$, so it has type $\tS$, proving (2).
\end{proof}

This description of lifts and the decorated version of the main theorem in \cite{Johnston_Birational-invariance} leads to the following formula.

\begin{proposition} \label{prop:birational-invariance-log-cylinders}
    Let $(\tX,\tD )\rightarrow (X,D)$ be a log-\'{e}tale modification of $(X,D)$, and let $\btau$ be a decorated tropical cylinder type in $\Sigma (X,D)$.
    Then
    \[k_{\tau}N_{\btau} = \sum_{\btau'} k_{\tau'} N_{\btau '} ,\]
    where the sum is over decorated tropical cylinder types $\btau'$ to $\tX$ that lift $\btau$.
\end{proposition}

\begin{proof}
    Similar to the proof of \cite[Corollary 9.4]{Johnston_Birational-invariance}, applying Theorem 1.4 of the cited reference we obtain 
    \[\frac{1}{m} \deg [\cM (X,\btau)]^{\vir} = \sum_{\btau'}\deg [\cM (\tX ,\btau' )]^{\vir},\]
    where the sum is over decorated lifts $\btau'$ of $\btau$ with $\dim \tau '= \dim \tau$ and $m$ is the index of the inclusion $(\tau')^{\gp}\rightarrow \tau^{\gp}$.
    We may restrict the sum to decorated tropical cylinder types that lift $\btau$, since otherwise the virtual degree of $\cM (\tX ,\btau' )$ is $0$ by \cref{lemma:tropical-lift-cylinder}.
    The index of the inclusion may be computed by applying the snake lemma to the following commutative diagram, where the rows are exact:
    \[
    \begin{tikzcd}
        0\ar[r] & \Lambda_{\tau_{v_i}'}  \ar[r]\ar[d,"(h_{\tau'})_{\ast}"] & \Lambda_{\tau_{v_i}}\ar[r]\ar[d,"(h_{\tau})_{\ast}"] & H \ar[r]\ar[d] & 0 \\
        0 \ar[r] & \Lambda_{\bsigma ' (v_i)} \ar[r] & \Lambda_{\bsigma (v_i)} \ar[r] & 0 .
    \end{tikzcd}
    \]
    We obtain  $m =\frac{k_{\tau}}{k_{\tau'}}$, where $k_{\tau}$ and $k_{\tau'}$ are defined in \eqref{eq:multiplicity-tropical-cylinder}.
    It remains to prove that if $\btau'$ is a tropical cylinder type lifting the tropical cylinder type $\btau$, then $\Aut (\btau' )\simeq \Aut (\btau )$.
    Let $\bomega_1,\bomega_2$ (resp. $\bomega_1',\bomega_2'$) be the decorated broken line types obtained by splitting $\btau$ (resp. $\btau'$) as in \cref{lemma:splitting-cylinder-type}(1).
    We have isomorphisms $\Aut (\btau )\simeq \Aut (\bomega_1)\times\Aut (\bomega_2)$ and $\Aut (\btau')\simeq \Aut (\bomega_1')\times\Aut (\bomega_2')$.
    Furthermore $\bomega_i'$ is a decorated tropical lift of $\bomega_i$ for $i=1,2$.
    Using the description of the automorphism group of a decorated broken line type in \cite[\S4.2]{Gross_Canonical-wall-structure-and-intrinsic-mirror-symmetry}, we are left to show that if $\bgamma'$ is a decorated wall type in $\tSigma$ lifting a decorated wall type $\bgamma$ to $\Sigma$, then $\Aut (\bgamma )\simeq \Aut (\bgamma')$.
    
    We denote by $L_{\out}$ (resp. $L_{\out}'$) the unique leg of $\bgamma$ (resp. of $\bgamma'$).
    Since $\bgamma'$ is a top dimensional lift of $\bgamma$, it corresponds to a choice of maximal cone in the subdivision of the image cone of $L_{\out}$ induced by $\tSigma$, and a choice of puncturing for the leg.
    These choices determine a unique subdivision of the graph $G_{\gamma}$ which produces $G_{\gamma'}$.
    If two edges of $G_{\gamma}$ are identified under an element of $\Aut (\bgamma)$, they are subdivided in the same way in $G_{\gamma'}$ as automorphisms commute with $\bsigma$ and $\bu$.
    Any automorphism of a wall type restricts to the identity on the shortest path connecting the unique leg to a vertex of valence at least $3$.
    This implies that $\Aut (\bgamma)$ injects in $\Aut (\bgamma' )$, and we can prove surjectivity in the same way.
    Hence $\Aut (\bgamma) \simeq \Aut (\bgamma')$, competing the proof.
\end{proof}

The following lemma identifies the total curve class of a lift of a decorated cylinder type.

\begin{lemma} \label{lemma:total-class-lift-cylinder}
    Let $\pi\colon (\tX,\tD)\rightarrow (X,D)$ be a log-\'{e}tale modification. 
    Let $\btau$ be a decorated cylinder type in $\Sigma$ with contact orders $\bu$ and total curve class $A$, and let $\btau'$ be a decorated cylinder type in $\tSigma$ lifting $\btau$.
    If $N_{\btau'}\neq 0$, then the total curve class of $\btau'$ is $\tA$.
\end{lemma}

\begin{proof}
    Let $A'$ denote the total curve class of $\btau'$, by the blowup formula $A'$ is uniquely determined by its intersection number with each exceptional component of $\pi$ and by the class $\pi_{\ast}A'$.
    Since $\btau'$ is a decorated lift of $\btau$, the latter is $A$.
    Since $N_{\btau'}\neq 0$, the moduli space $\cM (X ,\btau' )$ is not empty and we can use \cite[Corollary 2.30]{Abramovich_Punctured_logarithmic_maps} to compute the degree of $A'$ against a divisor in terms of the contact orders of the tropical type $\btau'$.
    We conclude that $A' = \tA$, where $\tA$ is defined after \cref{prop:moduli-spaces-before-after-blowup}.
    The proof is complete.
\end{proof}

\section{Comparison of cylinder counts}\label{sec:comparison-counts}

In this section we prove the main comparison result, \cref{thm:main-comparison}, which expresses non-archimedean cylinder counts in terms of log Gromov-Witten invariants. 
In \S\ref{subsec:algebraic-stacks-over-DVR}, we collect some general facts about irreducible components of algebraic stacks over a DVR, and in \cref{lemma:construct-substack} we explain how to produce a substack from an open substack in the special fiber.
In good situations, \cref{lemma:description-generic-fiber} describes the generic fiber of the constructed substack.
In \S\ref{subsec:family-moduli-space}, we precisely define the degenerate point constraint corresponding to \cref{fig:point-constraint}, and define a family of point-constrained moduli spaces which is parameterized by the spectrum of a complete DVR. 
In \S\ref{subsec:main-comparison}, we prove our main comparison (\cref{thm:main-comparison}) by first describing the special fiber of the moduli space (\cref{prop:decomposition-special-fiber}), and then refine the moduli spaces by selecting the tropical types corresponding to a given spine (Constructions \ref{constr:refined-tropical-moduli-space} and \ref{constr:refined-geometric-moduli-space}).
The description of the general fiber of this refined moduli space (\cref{prop:spine-refinement-generic-fiber}) relies on the study of logarithmic cylinder types in \cref{subsec:cylinder-counts}.

\subsection{Algebraic stacks over a DVR} \label{subsec:algebraic-stacks-over-DVR}

In this subsection, we fix the spectrum $S$ of a complete DVR of equal characteristic $0$ and provide some results and constructions related to algebraic stacks defined over $S$.
We denote by $s$ the closed point of $S$, and by $\eta$ its generic point.

Given an algebraic stack $\cX$, we denote by $\vert \cX\vert$ the underlying topological space. 
When talking about a substack $\cY \subset \cX$, we generally equip $\cY$ with the induced stack structure. 
If $\cX$ is locally noetherian, its \emph{irreducible components} are by definition the irreducible components of $\vert \cX \vert$.
Irreducible components are equipped with the \emph{reduced} stack structure, we view them as integral closed substacks of $\cX$.
The multiplicity of irreducible components is defined in \cite[Tag 0DR4]{Stacks_project}.
We recall that if $\cX$ is quasi-separated, then $\vert \cX\vert$ is a sober topological space.
We denote by $\Irr (\cX )$ the set of irreducible components of $\cX$.

The characterization of flatness for schemes over a DVR \cite[III, Proposition 9.7]{Hartshorne_Algebraic_geometry} extends to algebraic stacks. 

\begin{proposition}\label{lemma:flatness-over-DVR}
  Let $f\colon \cX\rightarrow S$ be a morphism from a locally noetherian, quasi-compact algebraic stack to the spectrum of a DVR.
  Then the following are equivalent:
  \begin{enumerate}[wide]
    \item $f$ is flat,
    \item $\vert \cX\vert$ equals the closure of the generic fiber $\vert \cX_{\eta}\vert$.
  \end{enumerate}
  In particular, if $f$ is flat then every irreducible component of $\cX$ dominates $S$.
\end{proposition}

\begin{proof}
  Fix a surjective smooth morphism $g\colon U\rightarrow \cX$, with $U$ an affine scheme.
  The map $\cX\rightarrow S$ is flat if and only if $U\rightarrow S$ is flat.
  The chart $g$ induces a smooth chart $U_{\eta}\rightarrow \cX_{\eta}$.
  Since $g(\overline{\vert U_{\eta}\vert}) = \overline{g(\vert U_{\eta}\vert )} = \overline{\vert \cX_{\eta}\vert}$, if the closure of $U_{\eta}$ equals $U$ then condition (2) is satisfied.
  Conversely, the preimage of a dense open set under a surjective open map is dense.
  Hence, if $\cX_{\eta}$ is dense in $\cX$, then $g^{-1} (\cX_{\eta} ) = U_{\eta}$ is dense in $U$.
  The equivalence of (1) and (2) is thus reduced to the scheme case.

  For the last statement, fix an irreducible component $\cX'$ of $\cX$.
  It is dominated by an irreducible component $U'\subset U$.
  Since $U\rightarrow S$ is flat, the generic point of $U'$ maps to the generic point of $S$.
  Thus, the same holds for $\cX'$ and the proof is complete.
\end{proof}

We use the following terminology for irreducible components.

\begin{definition}
    Let $\cX\rightarrow S$ be a locally noetherian, quasi-separated algebraic stack. 
    An irreducible component of $\cX$ is called \emph{horizontal} if it dominates $S$.
    Otherwise, it is called \emph{vertical}.
\end{definition}
\personal{
Irreducible components define a stratification, and by an open stratum we mean the a set of the form $\cZ\setminus (\cX \setminus\cZ)$, where $\cZ$ is an irreducible component.}

The following lemma provides a description of horizontal irreducible components. 
We will use it in the special case when $\cX_{\eta}$ is \'{e}tale.

\begin{lemma}\label{lemma:horizontal-irreducible-components-over-DVR}
    Let $\cX\rightarrow S$ be a locally noetherian, quasi-separated algebraic stack.
    There is a 1:1 correspondence between horizontal irreducible components of $\cX$ and irreducible components of $\cX_{\eta}$.  
\end{lemma}

\begin{proof}
    Up to replacing $\cX$ by the closure of $\cX_{\eta}$ in $\cX$, we may assume that $\Irr( \cX )$ consists only of horizontal components.
    The result then follows from the fact that a dense open inclusion between sober topological spaces induces a bijection on irreducible components. 
    More precisely, since $\vert \cX\vert$ and $\vert \cX_{\eta}\vert$ are sober we identify $\Irr (\cX )$ and $\Irr (\cX_{\eta} )$ with the sets of generic points of irreducible components. 
    Hence $\Irr (\cX_{\eta} )\subset \vert \cX_{\eta}\vert$, 
    and since components of $\cX$ are horizontal we also have $\Irr (\cX )\subset \vert \cX_{\eta} \vert$.

    We first show the inclusion $\Irr (\cX )\subset \Irr (\cX_{\eta} )$.
    Let $\xi\in \Irr (\cX )$, there exists $\xi'\in \Irr (\cX_{\eta} )$ such that $\xi\in \overline{\lbrace \xi'\rbrace}\cap \vert \cX_{\eta}\vert$ (where we mean closure in $\vert \cX\vert$).
    Then $\overline{\lbrace \xi\rbrace}\subset \overline{\lbrace \xi'\rbrace}$, and since $\overline{\lbrace \xi\rbrace}$ is an irreducible component this inclusion is an equality. 
    Since $\vert\cX\vert$ is sober, we must have $\xi = \xi'$.

    For the reverse inclusion, we prove that if $\xi\in\Irr (\cX_{\eta} )$ then $\overline{\lbrace \xi\rbrace}$ is an irreducible component of $\vert \cX\vert$.
    Since $\overline{\lbrace \xi\rbrace}$ is irreducible, there exists $\xi'\in \Irr (\cX )$ such that $\overline{\lbrace \xi\rbrace}\subset \overline{\lbrace \xi '\rbrace}$.
    Intersecting with $\vert \cX_{\eta}\vert$, we obtain the same inclusion for the closures of $\xi$ and $\xi'$ in $\vert \cX_{\eta}\vert$.
    But now, both are irreducible components of $\vert \cX_{\eta}\vert$ hence the inclusion is an equality, and we conclude that $\xi =\xi'$ using that $\vert \cX_{\eta}\vert$ is sober. 
    The lemma is proved.
\end{proof}

Given an algebraic stack $\cX$ over a DVR, the next construction associates a proper substack of $\cX$ to an open substack of the special fiber, under some properness assumptions.
We will apply this construction in \cref{subsubsec:spine-refinement} to prove our main comparison result.

\begin{construction}\label{lemma:construct-substack}
    Let $f\colon \cX\rightarrow S$ be an algebraic stack, with $f$ separated and locally of finite type, and $\cX_s$ proper.
    Given an open substack $\cU\subset \cX_s$, we construct a proper sustack $\cX (\cU )\subset \cX$ such that  $\cX (\cU )_s$ equals the closure of $\cU$ in $\cX_s$.

    Completion along the special fiber produces a morphism of formal stacks $\fX\rightarrow \hS$, where $\fX$ (resp. $\hS$) is the completion of $\cX$ (resp. $S$) along the special fiber.
    Note that $\hS$ is the formal spectrum of a complete DVR.
    The open substack $\cU\subset \cX_s$ produces an open substack $\fU\subset \fX$.
    The closure of $\fU$ is a closed substack of $\fX$, whose underlying topological space coincides with the closure of $\cU$ in $\cX_s$.
    Since $\cX_s$ is proper, the formal stack $\fX$ is proper over $\hS$.
    Then the closure $\overline{\fU}\subset \fX$ is proper over $\hS$, and by formal GAGA for stacks (\cite[Corollary 4.6]{Conrad_Formal-GAGA}) there exists a proper substack $\cX (\cU )\subset \cX$ whose formal completion coincides with $\overline{\fU}$.
    This stack $\cX (\cU )$ satisfies the desired properties.
\end{construction}

The following lemma describes the generic fiber of the constructed stack $\cX (\cU )$ under an additional assumption.

\begin{lemma} \label{lemma:description-generic-fiber}
  Let $\cX\rightarrow S$ be a proper algebraic stack, $\cU \subset \cX_s$ an open substack with closure $\overline{\cU}$ in $\cX_s$, producing the proper substack $\cX (\cU )\subset\cX$ as in \cref{lemma:construct-substack}.
  Assume that no horizontal irreducible component of $\cX$ intersects $\overline{\cU}\setminus\cU$.
  Then $\cX (\cU )$ is the substack given by the union of $\overline{\cU}$ and the horizontal irreducible components of $\cX$ which intersect $\cU$.
\end{lemma}

\begin{proof}
    Let $\cZ\subset \cX$ be a substack associated to a horizontal irreducible component.
    Since $\cZ_s$ is connected by \cref{lemma:connected-special-fiber}, if $\cZ$ intersects both $\cU$ and $\cX_s\setminus \overline{\cU}$ then necessarily $\cZ$ must intersect the complement of $\cU\cup (\cX_s\setminus \overline{\cU})$ in the special fiber, which is $\overline{\cU}\setminus\cU$.
    This contradicts the assumption, and we deduce that the special fiber of every horizontal irreducible component is contained either in $\cU$, or in $\cX_s\setminus\overline{\cU}$.

    Now, let $\cZ$ be the substack defined by the horizontal irreducible components which intersect $\cU$.
    To prove the lemma, it is enough to show that $\cX (\cU )_{\eta} = \cZ_{\eta}$.
    If $\xi$ is a point of $\cX (\cU )_{\eta}$, then its closure in $\cX$ has special fiber contained in $\overline{\cU}$, hence in $\cU$ by the assumption of the lemma.
    It follows that $\xi$ lies on a horizontal irreducible component which intersects $\cU$, hence it is a point of $\cZ_{\eta}$.
    Conversely, note that $\cZ$ is a proper substack of $\cX$.
    Since $\cZ_s \subset \cX(\cU)_s$, formal GAGA implies that $\cZ\subset \cX (\cU )$.
    In particular $\cZ_{\eta}\subset \cX (\cU )_{\eta}$, completing the proof. 
\end{proof}

\begin{lemma} \label{lemma:connected-special-fiber}
    Let $\cX\rightarrow S$ be a proper algebraic stack.
    Then there is a bijection $\pi_0 (\cX_s ) \simeq \pi_0 (\cX )$.
\end{lemma}

\begin{proof}
    Let $\fX$ denote the completion of $\cX$ along the special fiber. 
    Since $\vert \fX \vert = \vert \cX_s\vert$, we have a bijection $\pi_0 (\cX_s)\simeq \pi_0 (\fX )$.
    On the other hand, by formal GAGA we have $\pi_0 (\fX )\simeq \pi_0 (\cX )$.
    Since there are finitely many connected components, this can be seen by looking at the rank of $H^0 (\cX , \cF ) \simeq H^0 (\fX ,\widehat{\cF} )$ for a constant sheaf $\cF$.
\end{proof}

\subsection{Construction of the family moduli space}  \label{subsec:family-moduli-space}
In this subsection, we construct a family of point constrained moduli spaces (\cref{def:family-moduli-space}), using a very degenerate point constraint in the special fiber (see \cref{fig:point-constraint}).

\subsubsection{Degenerate point constraint}

We review the key features of the degenerate point constraint in \cite{Johnston_Comparison}, itself adapted from \cite[\S 6]{Gross_Canonical-wall-structure-and-intrinsic-mirror-symmetry}.

The point constraint is encoded by a morphism of fs log schemes
\begin{equation} \label{eq:point-constraint}
  T \longrightarrow X,
\end{equation}
where $T = (\underline{T} , \cM_T )$ with $\underline{T} = \Spec \cO_{U,0}$ for $U$ an \'{e}tale neighborhood of $0\in\bbA_{\bbk}^1$, and $\cM_T$ an fs log structure with $R\coloneqq \ghost{\cM}_{T,s}$ a free monoid of rank $d = \dim X$ and $\ghost{\cM}_{T,\eta} = 0$.
Geometrically, the generic fiber $T_{\eta}$ maps to a point of $X\setminus D$ which we can take as generic as needed.
The special fiber $T_s$ maps to a $0$-dimensional stratum of $D$ associated to a maximal cone $\sigma_0\in \Sigma (X,D)$.
We choose $\sigma_0$ in the essential part $\Sigma^{\ess} \subset \Sigma(X,D)$.
Tropically, the map $\Sigma (T)\rightarrow \Sigma (X,D)$ is an inclusion $\bbR_{\geq 0}^{d}\hookrightarrow \sigma_0$ whose image can be taken as small as we want.
This fact is used to achieved transversality with respect to the evaluation map and obtain a flat family of moduli spaces (see \cref{prop:choice-point-constraint}).

We now modify the point constraint to fit the set-up of \cref{subsec:log-setup}.
We let $k^{\circ}$ denote the completion of $\cO_{U,0}$, it is now a complete regular DVR smooth over $\bbk$.
\personal{In particular $k^{\circ}\simeq \tk\dbb{t}$ by the classification of complete regular local rings.}
We replace $\underline{T}$ by $\Spec k^{\circ}$, equipped with the pullback log structure under the morphism $\Spec k^{\circ}\rightarrow \underline{T}$.
Now, the point constraint \eqref{eq:point-constraint} is parametrized by the spectrum of a complete DVR, smooth over $\Spec\bbk$.

Let $B$ denote the scheme $\underline{T}$ equipped with the trivial log structure, we denote by $\eta = \Spec k$ (resp. $s = \Spec \tk$) the generic point (resp. the closed point).
The generic fiber of the log scheme $T$ is the trivial log point $T_{\eta} = \Spec k$, while the special fiber is the log point $\Spec (R\rightarrow \tk)$.
As in \cref{sec:preliminaries}, we define 
\begin{align*}
    \underline{X_B} \coloneqq \underline{X}\times_{\bbk} \underline{B} ,\\
    \underline{D_B} \coloneqq \underline{D}\times_{\bbk} \underline{B},
\end{align*}
and equip $X_B$ with the divisorial log structure coming from $D_B$.
We note that $\Sigma (X_B,D_B) = \Sigma (X_s,D_s) = \Sigma (X,D)$, in particular $\Sigma (X_B,D_B)$ is $d$-dimensional and not $(d+1)$-dimensional, because $D\times B$ does not have a $0$-stratum.

We have a morphism of log schemes $T\rightarrow B$, and obtain a commutative diagram of fs log schemes:
\begin{equation} \label{eq:point-constraint-family}
\begin{tikzcd}
    T \ar[r] \ar[dr] & X_B \ar[d] \\
    & B .
\end{tikzcd}
\end{equation}

\subsubsection{Point constrained moduli space}

We can now recast the construction of the family of moduli spaces in \cite{Johnston_Comparison} into a version for log stable maps in $X_B$.
Since the log structure on $B$ is trivial, the Artin fan of $B$ is $\cA_B \simeq \Spec \bbk$.
Then, the relative Artin fan for $X_B$ is simply
\[\cX_B \coloneqq B\times_{\cA_B} \cA_X \simeq B\times \cA_X.\]

\begin{notation}[Choice of contact orders] \label{notation:choice-contact-orders}
    Let $A\in \NE (X,\bbZ )$.
    We fix a vertex type $\beta$ of tropical maps to $\Sigma (X,D)$ i.e. a tropical type whose underlying graph has a single vertex and no edges, with the following properties.
    The total curve class of $\beta$ is $A$, and $\beta$ has three legs $\lbrace L_1,L_2, L_i\rbrace$ such that: 
    \begin{enumerate}[wide]
        \item $\bsigma (L_1),\bsigma (L_2)\in\Sigma^{\ess}$,
        \item $\bu (L_1)\in \bsigma (L_1)\setminus \lbrace 0\rbrace$ and $\bsigma (L_2)\in \bsigma (L_2)\setminus \lbrace 0\rbrace$, and
        \item $\bu (L_i ) =0$.
    \end{enumerate}
    The type $\beta$ induces a global type for $\Sigma (X_B,D_B)\simeq \Sigma (X,D)$.
\end{notation}
We introduce the moduli spaces
\[\cM (X_B/B , \beta ) ,\quad \fM (\cX_B/B ,\beta ) \]
of basic logarithmic stable maps in $X_B/B$ (resp. $\cX_B/B$) of type $\beta$.
The latter prametrizes basic tropical maps to $\Sigma (X_B,D_B)$ of type $\beta$, and we refer to it as the tropical moduli space.
We note that by \cite[Proposition 5.13]{Abramovich_Punctured_logarithmic_maps}, we have $\cM (X_B/B ,\beta )\simeq \cM (X_B /\Spec\bbk,\beta )$ and similarly for the tropical moduli space.
\personal{Henceforth, we will omit $B$ from the notation.}
Since $B$ is log smooth over $\Spec\bbk$, the general theory developed in \cite{Abramovich_Punctured_logarithmic_maps} shows the following. 

\begin{proposition}
    Fix $\beta$ as in \cref{notation:choice-contact-orders}.
    \begin{enumerate}[wide]
        \item The tropical moduli space $\fM (\cX_B/B,\beta )$ is not empty, and is reduced and pure dimensional of relative dimension $0$ over $B$.
        \item The structural map  $\cM (X_B/B,\beta )\longrightarrow B$ is proper.
        \item The map $\cM (X_B/B,\beta) \rightarrow \fM (X_B/B,\beta )$ admits a relative perfect obstruction theory.
    \end{enumerate}
\end{proposition}

We have evaluation maps at the contact order zero point
\begin{align*}
  \ev_{X_B} \colon &\cM (X_B/B , \beta )  \longrightarrow  X_B, \\
  \ev_{\cX_B}\colon & \fM (\cX_B/B ,\beta ) \longrightarrow  \cX_B 
\end{align*}
which are compatible with the log structures and commute with the projection to $B$.
We form the fiber product:
\begin{equation} \label{eq:evaluation-stack}
    \fM^{\ev} (\cX_B/B ,\beta) \coloneqq \fM (\cX_B/B ,\beta ) \times_{\underline{\cA_{X_B}}} \underline{X_B}.
\end{equation}
It coincides with the fs fibered product by the following lemma.

\begin{lemma} \label{lemma:evaluation-stack-is-fine-pullback}
  There is a cartesian diagram in all categories
  \[\begin{tikzcd}
    \fM^{\ev} (\cX_B/B ,\beta )\ar[r] \ar[d] & X_B \ar[d] \\
    \fM (\cX_B/B ,\beta ) \ar[r] & \cX_B .
  \end{tikzcd}\]
\end{lemma}

\begin{proof}
  The vertical maps are strict, and the diagram is cartesian at the level of stacks.
  The result then follows from \cite[Remark 2.1.3]{Ogus_Lectures-log-geometry}.
\end{proof}

The generic fiber of $\cM (X_B/B,\beta )$ is identified with $\cM (X_{\eta} , \beta )$, which parametrizes log stable maps to the base changed variety $X_{\eta}\coloneqq X\times_{\bbk} k$.
The Artin fan $\cX_{\eta}$ of $X_{\eta}$ is isomorphic to $\cX_B\times_B\eta \simeq\cA_{X}\times_{\bbk} k$.
Thus, the generic fiber of the tropical moduli spaces $\fM (\cX_B/B,\beta )$ and $\fM^{\ev} (\cX_B/B,\beta )$ are identified with $\fM (\cX_{\eta},\beta )$ and $\fM^{\ev} (\cX_{\eta} ,\beta )$ (defined analogously to \eqref{eq:evaluation-stack}).
In a similar way, denoting by $\cX_s$ the Artin fan of $X_s  =X\times_{\bbk}\tk$, the special fibers are isomorphic to the $\tk$-stacks $\cM (X_s ,\beta )$, $\fM (\cX_s , \beta )$, and $\fM^{\ev} (\cX_s ,\beta )$.

\begin{definition} \label{def:family-moduli-space}
    The \emph{point constrained moduli spaces} associated to \eqref{eq:point-constraint-family} are
    \begin{align*}
      \cM (X_B , \beta )_T&\coloneqq \cM (X_B,\beta )\times_{X_B}^{\fs} T ,\\
      \fM^{\ev}  (\cX_B ,\beta )_T &\coloneqq \fM^{\ev} (\cX_B ,\beta ) \times_{X_B}^{\fs} T.
    \end{align*}
\end{definition}

\begin{proposition} \label{prop:choice-point-constraint}
    There exists a choice of point constraint in \eqref{eq:point-constraint-family} such that:
    \begin{enumerate}[wide]
        \item $\fM^{\ev} (\cX_B/B ,\beta )_T\rightarrow T$ is flat, and
        \item $\cM (X_{\eta} , \beta )_{T_{\eta}}$ is \'{e}tale over $\eta$, and contained in the image of the open embedding $\cM^{\sm} (U,\beta )\subset \cM (X_{\eta} ,\beta )$ (see \cref{prop:embedding-smooth-locus}).
    \end{enumerate}
\end{proposition}

\begin{proof}
    By base change, the map in (1) is log flat.
    If we can choose the point constraint so that the map is also integral, then it will be flat (\cite[Proposition 2.3]{Gross_Intrinsic_mirror_symmetry}).
    Since the log structure on $T$ is free, integrality can be reduced to the tropical criterion of \cite[Theorem 2.9]{Gross_Intrinsic_mirror_symmetry}, which involves the map $\Sigma (T)\rightarrow \Sigma (X_B,D_B)$ and the tropical evaluation map $\Sigma (\ev_{\cX_B} )$.
    The criterion can be achieved by shrinking the image of the tropical point constraint, as in \cite[Step 3]{Johnston_Comparison}.

    (2) is proved at the end of the proof of Theorem 1.1 in \cite{Johnston_Comparison}, provided that the image of $T_{\eta}$ in $X_{\eta}$ is generic enough. 
\end{proof}

\subsection{Main comparison} \label{subsec:main-comparison}
We refine the analysis of the special fiber in \cite{Johnston_Comparison} to obtain a correspondence between non-archimedean spine counts and logarithmic invariants.
We continue to fix a decorated vertex type $\beta$ as in \cref{notation:choice-contact-orders}, with total curve class $A\in \NE (X,\bbZ )$.
To simplify notations, we write 
\begin{align*}
    \cM_T &\coloneqq \cM (X_B/B ,\beta)_T ,\\
    \cM_{\eta}&\coloneqq \cM (X_{\eta} , \beta)_{T_{\eta}} ,\\
    \cM_s &\coloneqq \cM (X_s , \beta )_{T_s},
\end{align*}
and similarly we denote by $\fM_T$, $\fM_{\eta}$, and $\fM_s$ the tropical moduli spaces.
We have constructed a diagram where all squares are fs-cartesian
 \begin{equation} \label{cd:family-moduli-space}
 \begin{tikzcd}
   \cM_s \ar[r]\ar[d, "\varepsilon_s"] & \cM_T \ar[d,"\varepsilon_T"] & \cM_{\eta} \ar[d, "\varepsilon_{\eta}"] \ar[l] \\
   \fM_s \ar[r] \ar[d] & \fM_T \ar[d] & \fM_{\eta} \ar[l] \ar[d] \\
   s \ar[r] & T & \eta \ar[l].
 \end{tikzcd}
 \end{equation}
 The bottom vertical arrows are flat, all the vertical arrows are proper, and the map $\cM_{\eta} \rightarrow \eta$ is étale.
 The map $\varepsilon_T$ admits a relative perfect obstruction theory, providing a relative virtual fundamental class $[\cM_T /T]^{\vir}\coloneqq \varepsilon_T^! [\fM_T]$ of degree $0$ relative to $T$.
 By pullback, we obtain degree $0$ virtual fundamental classes $[\cM_s]^{\vir} \coloneqq \varepsilon_s^! [\fM_s]$ and $[\cM_{\eta}]^{\vir}\coloneqq \varepsilon_{\eta}^! [\fM_{\eta}]$.
 Since $\cM_{\eta}$ is smooth, the latter coincides with the usual fundamental class. 
 By invariance of the virtual fundamental class under flat deformations, we have 
 \begin{equation}
     \deg [\cM_{\eta}]^{\vir} = \deg [\cM_s]^{\vir} .
 \end{equation}

\subsubsection{Description of the special fiber}

We have the following description of irreducible components of $\fM_s$ by tropical types.

 \begin{lemma} \label{prop:tropical-type-special-fiber}
   Let $\xi$ be the generic point of an irreducible component $\fM_{\xi}$ of $\fM_s$, let $\tau$ be the associated tropical type.
   Assume that $\deg \varepsilon_s^! [\fM_{\xi}]\neq 0$. 
   Then $\tau$ is a cylinder type with $\bsigma (v_i)  =\sigma_0$, and the multiplicity $\mu_{\xi}$ of $\fM_{\xi}$ only depends on $\tau$ and is
   \[ \mu_{\tau} \coloneqq \vert \coker (R^{\gp}\rightarrow Q_{\xi}^{\gp})\vert ,\]
     where $R = \ghost{\cM}_{T,s}$ is the stalk of the ghost sheaf of the log structure for the point constraint at $s$, and $Q_{\xi}$ is the stalk of the ghost sheaf of $\fM_s$ at $\xi$.
 \end{lemma}

 \begin{proof}
     The fact that $\tau$ is a cylinder type follows from the choice of tropically transverse point constraint as in \cite[Theorem 6.4, Steps 2-3]{Gross_Canonical-wall-structure-and-intrinsic-mirror-symmetry} and \cite[Step 4]{Johnston_Comparison}.
     We deduce the formula for the multiplicity of $\fM_{\xi}$ as in \cite[Theorem 6.4, Step 3]{Gross_Canonical-wall-structure-and-intrinsic-mirror-symmetry}.
 \end{proof}

For a decorated tropical cylinder type $\btau$ of total curve class $A$, we introduce the proper substack 
\[\fM_s (\btau )\subset \fM_s\]
given by the union of $\fM_{\xi}$ for $\xi$ a generic point with associated type $\btau$, viewed as a \emph{non-reduced} substack of $\fM_s$.
The following is a first refinement of the main result of \cite{Johnston_Comparison}, in the case of $3$-pointed invariants. 

\begin{proposition} \label{prop:decomposition-special-fiber}
    Let $\btau$ be a decorated tropical cylinder type of total curve class $A$ and with $\bsigma (v_i ) =\sigma_0$.
    Then
    \[\deg \varepsilon_s^! [\fM_s (\btau )] = k_{\tau}N_{\btau} .\]
    In particular, summing over those types gives:
    \[\deg [\cM_{\eta}] = \sum_{\btau} k_{\tau} N_{\btau}.\]
\end{proposition}

\begin{proof}
    The proof is analogous to \cite[Theorem 6.4, Step 4]{Gross_Canonical-wall-structure-and-intrinsic-mirror-symmetry}.
\end{proof}

\subsubsection{Refinement by spine} \label{subsubsec:spine-refinement}

Recall the wall structure $\Wall_A$ from \cref{constr:total-walls}.
Fix a transverse $\Wall_A$-transverse spine $S$ to $\Sk (X,D)\simeq \Sigma$, with contact orders specified by $\beta$ and such that $\Sp^{-1} (S)\cap \cM_{\eta}\neq \emptyset$.
In particular the non-archimedean cylinder count $N(S,A)$ is nonzero.

 We will define refined moduli spaces that fit into a diagram analogous to \eqref{cd:family-moduli-space}, and only contain stable maps related to the cylinder spine $S$.
 Before defining those refined moduli spaces, we modify the stacks $\fM_T$ and $\cM_T$ as follows.
 \begin{enumerate}[wide]
     \item We replace $\fM_T$ with the closure of the complement of $\fM_{\xi}$ for $\deg \varepsilon_s^![\fM_{\xi}] = 0$ and $\fM (\btau )$ for $N_{\btau} = 0$.
     \item We replace $\cM_T$ by the irreducible closure of the preimage of the newly constructed $\fM_T$, and then consider the closure of the complement of the horizontal irreducible components corresponding to points $[f]\in\cM_{\eta}$ with $\Sp (f^{\an} )$ non transverse.
 \end{enumerate}

The following construction provides the relevant tropical moduli space for the spine $S$.
 
\begin{construction}[Refined tropical moduli space]  \label{constr:refined-tropical-moduli-space}
     We define a proper substack 
     \begin{align}\label{eq:spine-refined-trop-moduli}
         \fM (S)_s\coloneqq \bigcup_{\btau} \fM_s (\btau ) \subset \fM_s,
     \end{align}
     where the union is over decorated tropical cylinder types $\btau$ of total curve class $A$ with $\bsigma (v_i)= \sigma_0$, $\Sp (\tau )$ of type $S$, and $N_{\btau}\neq 0$.
     We introduce the (non-reduced) proper substack
     \[\fM (S)_T\subset \fM_T\]
     given by horizontal components of $\fM_T$ which intersect $\fM (S)_s$.
     By \cref{lemma:flatness-over-DVR}, the morphism $\fM_T\rightarrow T$ is flat.
 \end{construction}

 We define analogous substacks at the geometric level (see \cref{fig:refined-moduli-space}).

 \begin{construction}[Refined geometric moduli space] \label{constr:refined-geometric-moduli-space}
    We associate to the spine $S$ the proper substack
    \begin{align}
        \cM (S)_s\subset \cM_s
    \end{align}
    whose underlying topological space is given by irreducible components of $\vert \cM_s\vert$ that intersect $\varepsilon_s^{-1} (\fM(S)_s)$, but equipped with the non-reduced substack structure.
    Consider the open substack 
    \[\cU (S)\coloneqq \cM (S)_s\setminus \bigg(\overline{\cM_s\setminus \cM (S)_s}\bigg),\]
    corresponding to removing the intersection of $\cM (S)_s$ with irreducible components of $\cM_s$ that do not intersect $\varepsilon_s^{-1} (\fM (S)_s )$.
    Applying \cref{lemma:construct-substack}, we obtain a proper substack
    \begin{align}
        \cM (S)_T\subset \cM_T
    \end{align}
    whose special fiber is $\overline{\cU (S)} = \cM (S)_s$.
    We denote by $\cM (S)_{\eta}$ the general fiber of $\cM (S)_T$, it is a proper substack of $\cM_{\eta}$.
 \end{construction}

\begin{figure}
    \centering
    \setlength{\unitlength}{0.45\textwidth}
    \begin{picture} (0.5,1)
    %\includesvg[width=0.3\linewidth]{Moduli-space.svg}
    \put(0,0.03){\includegraphics[height=0.97\unitlength]{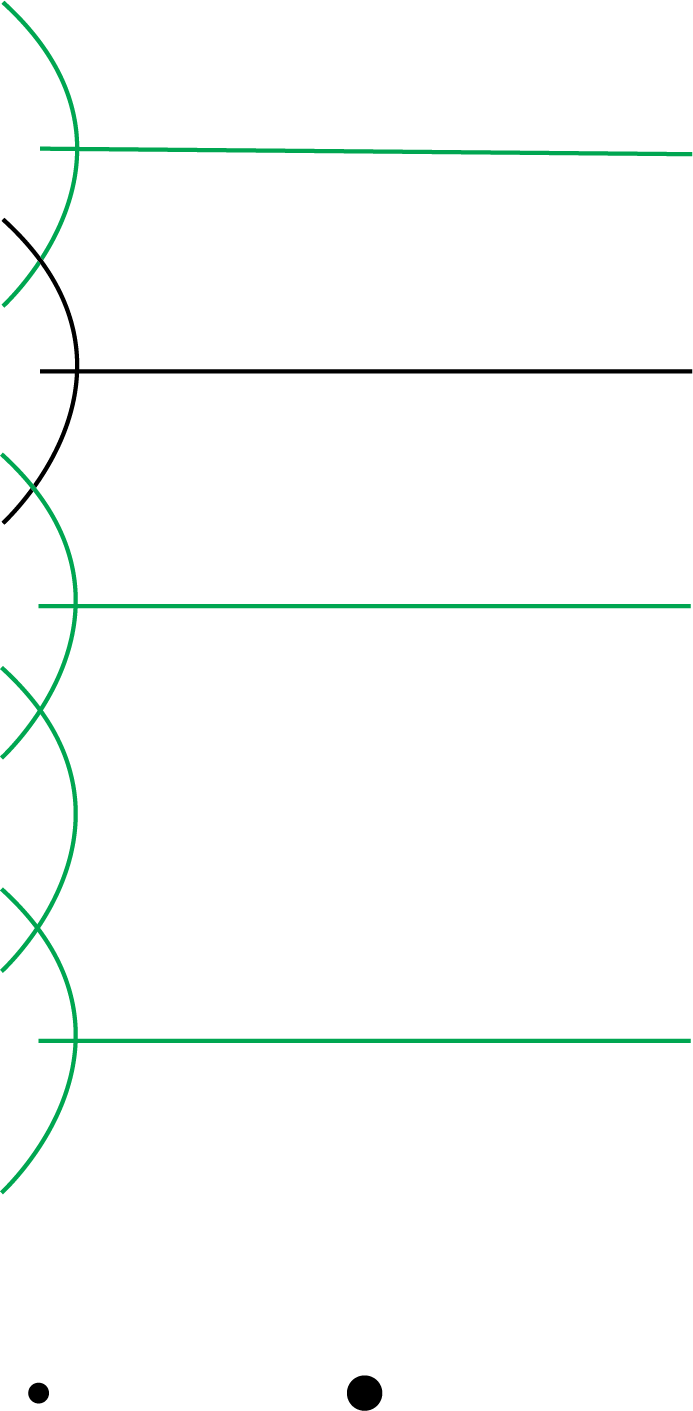}}
    \put(0.015,0.0){$s$}
    \put(0.235,0){$\eta$}
    \put(0.15,0.95){\textcolor[RGB]{0,166,81}{$\cM(S)_T$}}
    \end{picture}
    \caption{The refined geometric moduli space $\cM (S)_T$ is obtained by selecting vertical components in the special fiber, and considering the horizontal components which intersect them. }
    \label{fig:refined-moduli-space}
\end{figure}
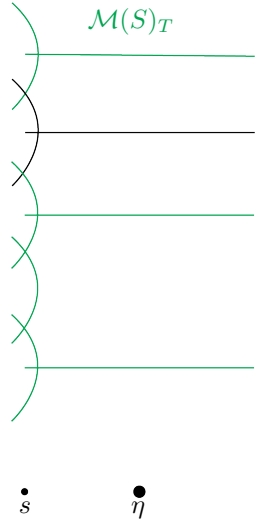

 The next proposition describes the generic fiber of the refined moduli space.

 \begin{proposition} \label{prop:spine-refinement-generic-fiber}
   We have 
   \[\cM (S)_{\eta}^{\an} = \Sp^{-1} (S)\cap \cM_{\eta}^{\an}.\]
 \end{proposition}

 \begin{proof}
     First, we identify the generic fiber $\cM (S)_{\eta}$ using \cref{lemma:description-generic-fiber}.
     Let $\cZ\in \Irr (\cM_T )$ be a horizontal component, and fix a closed point $x$ in $\cZ\cap\cM_s$.
     Let $\tau'$ denote the tropical type associated to $x$.
     If $x$ lies in $\cM(S)_s\setminus \cU(S)$, then $\Sp (\tau' )$ is a spine type that interpolates between the types of $S$ and another spine $S'$, distinct from $S$.
     This contradicts \cref{lemma:degeneration-transverse-spine}, because $S$ is transverse.
     We deduce that no horizontal irreducible components of $\cM_T$ intersects $\cM (S)_s\setminus \cU (S)$.
     So, the generic fiber $\cM (S)_{\eta}$ is described by \cref{lemma:description-generic-fiber}.

    Next, we prove each inclusion.
    Keeping the same notations, let $[f]\in\cM_{\eta}$ denote the generic point of $\cZ$.
    The specialization $[f]\rightsquigarrow x$ corresponds to a morphism $V\rightarrow \cM_T$, with $V$ the spectrum of a complete DVR.
    By \cref{prop:spine-tropical-type-special-fiber}, the type $\Sp (\tau' )$ coincides with the type $\Sp (f^{\an} )$.
    By \cref{prop:tropical-type-special-fiber} and \cref{prop:transversality-cylinder-type}, the type $\Sp (\tau ')$ also coincides with the type $S'$ of a $\Wall_A$-transverse spine.
    In particular $\Sp (f^{\an})$ and $S'$ have the same type. 
    \personal{Here, we use that we have removed 1/ Vertical components $\fM_{\xi}$ with degree $0$, 2/ Components $\fM (\btau )$ with $N_{\btau} = 0$ and 3/ Points of $\cM_{\eta}$ with non-transverse spines.}
    
    We deduce that if $f^{\an}$ has spine $S$, then $\cZ$ intersects $\cM (S)_s$, and hence $\cZ_{\eta}\subset \cM (S)_{\eta}$.
    Conversely, if the generic point of $\cZ$ lies in $\cM (S)_{\eta}$ then $\Sp (f^{\an} )$ and $S$ are two $\Wall_A$-transverse spines that have the same image under the map $(\dom ,\ev_i )$ taking the simplification of domain and evaluating at the interior marked point.
    It follows from the rigidity property of transverse spines (\cref{prop:rigidity-transverse-spines}) that $S = \Sp (f^{\an} )$, proving the second inclusion. 
    The proof is complete.
 \end{proof}

Note that the proof of \cref{prop:spine-refinement-generic-fiber} relies in an essential way on the rigidity of cylinder spines established in \cref{lemma:degeneration-transverse-spine}.
We can now prove the main comparison.

\begin{theorem}[Main comparison] \label{thm:main-comparison}
    Let $U$ be a connected smooth affine log Calabi-Yau variety satisfying the assumptions of \cref{subsec:log-setup}.
    Let $S$ be a cylinder spine in $\Sk (U)$ of type $\beta = (A,\bu )$.
    Then the non-archimedean cylinder count $N(S,A)$ decomposes as a sum of logarithmic cylinder counts
    \[N(S,A ) = \sum_{\btau} k_{\tau} N_{\btau} , \]
    where the sum is over decorated tropical cylinder types $\btau$ of total curve class $A$, contact orders specified by $\beta$, and $\Sp (\tau )$ of type $S$ (see \cref{def:spine-type} for $\Sp (\tau )$).
\end{theorem}

\begin{proof}
    We first prove the formula after replacing $(X,D)$ by a toric blowup $(\tX,\tD)$ in which the non-archimedean count $N(S,A)$ is defined.
    By \cref{lemma:transverse-spine-after-blowup}, after deforming $S$ we can guarantee that the spine $\tS$ induced by the subdivision of $\Sigma$ is still a cylinder spine, of type $\tbeta = (\tA, \bu )$.
    Since the count $N(S,A)$ is deformation invariant amongst transverse spines (\cite[Theorem 8.12]{Keel-Yu_Log-CY-mirror-symmetry}), this does not change the non-archimedean counts.
    We can now perform all the previous constructions with respect to $(\tX,\tD)$, using the spine $\tS$ and the data $\tbeta = (\tA , \bu )$.
    
    By construction, the diagram \eqref{cd:family-moduli-space} restricts to a cartesian diagram for the refined moduli spaces. 
    Since $\fM(\tS)_T\rightarrow T$ is flat, we have 
    \[\deg [\cM (\tS)_{\eta} ]^{\vir} = \deg [\cM (\tS)_s]^{\vir} .\]
    The left-hand side equals $N(\tS,\tA )$ by \cref{prop:spine-refinement-generic-fiber} and \eqref{eq:def-analytic-cylinder-counts}.
    By construction, the right-hand side has the decomposition 
    \[[\cM (\tS)_s]^{\vir} = \sum_{\btau'} \varepsilon_s^! [\fM (\btau' ) ],\]
    where the sum is over decorated tropical cylinder types to $\Sigma (\tX ,\tD )$ of total curve class $\tA$, contact orders specified by $\tbeta$, and $\Sp (\tau' )$ of type $\tS$, with the additional condition $N_{\btau'}\neq 0$.
    Taking the virtual degree, we obtain
    \begin{align*}
        N(\tS,\tA ) &= \sum_{\btau'} k_{\tau'}N_{\btau'} \\
        &= \sum_{\btau} \sum_{\btau'\rightarrow\btau} k_{\tau'}N_{\btau'} \\
        &= \sum_{\btau} k_{\tau}N_{\btau} ,
    \end{align*}
    where the sums over $\btau$ are indexed by cylinder types as in the statement of the theorem. 
    We used \cref{prop:decomposition-special-fiber} for the first equality, \cref{lemma:tropical-lift-cylinder,lemma:total-class-lift-cylinder} for the second equality, and \cref{prop:birational-invariance-log-cylinders} for the last equality. 
    Since $N(\tS,\tA ) = N(S,A)$ by definition, the theorem is proved.
\end{proof}

\section{Application to log Calabi-Yau surfaces} \label{sec:application-scattering-diagram}
In this section, we fix a connected smooth affine log Calabi-Yau surface $U$ over $\bbk$, and an snc compactification $(X,D)$ satisfying the assumptions of \cref{subsec:log-setup}.
The main result is the \emph{exponential formula} (\cref{thm:exponential-formula}), which asserts that the non-archimedean and logarithmic scattering diagrams coincide, and \cref{coro:mirror-algebras} which establishes the equivalence of the non-archimedean and logarithmic mirror constructions.
Both results are a consequence of the remarkable formula in \cref{prop:infinitesimal-cylinder-counts-to-log-counts}, which expresses non-archimedean infinitesimal cylinder counts in terms of punctured log Gromov-Witten invariants. 

In the surface case, the pair $(X,D)$ is also called a Looijenga pair, and the affineness assumption says that this pair is positive. 
In particular \cref{assumption:semipositivity} is always satisfied by \cite[Lemma 6.9]{Gross_Mirror_symmetry_for_log_Calabi-Yau_surfaces_I_published}.
Such a Looijenga pair admits a toric model, i.e. there exists a toric variety $X_t$ which is related to $X$ by a zigzag of toric blowups:
\[
\begin{tikzcd}
    & X_1 \ar[ld] \ar[rd] & & X_N \ar[ld] \ar[rd] & \\
    X & & \cdots & & X_t.
\end{tikzcd}
\]
The dual intersection fan $\Sigma$ of $(X,D)$ then has support isomorphic to $\bbR^2$ (see \cite[\S 2]{Keel_Yu_The_Frobenius}).
The affine structure on $\Sk (U)$ defined in \cref{subsec:log-setup} does not coincide with the affine structure defined by $\bbZ^2\subset \bbR^2$ in general, and unless $(X,D)$ is toric it has a singularity at the point $0\in\Sk (U)$, corresponding to the origin of $\bbR^2$.
Under the homeomorphism $\Sk (X,D)\simeq \vert \Sigma\vert$, the data of a wall structure in $\Sigma$ reduces to a collection of lines through the origin $0\in \Sk (U)$ and contained in $\Sk (U)\simeq \vert \Sigma^{\ess}\vert$, decorated with an integral tangent vector.

\subsection{Infinitesimal cylinder counts}

We discuss infinitesimal cylinder counts in the non-archi\-me\-dean setting. 
In general, we expect those counts to be related to \emph{punctured} log Gromov-Witten invariants.
Similar to the definition of non-archimedean cylinder counts in \cref{subsec:cylinder-counts}, infinitesimal cylinder counts are $3$-pointed non-archimedean Gromov-Witten invariants defined by fixing constraints on the spine and taking the degree of the evaluation map. 
The key difference is that infinitesimal cylinder counts are meant to provide a count of \emph{curves with boundary} inside $X$.
This is done in \cite{Yu_Enumeration_of_holomorphic_cylinders_I,Yu_Enumeration_of_holomorphic_cylinders_II,Keel_Yu_The_Frobenius,Keel-Yu_Log-CY-mirror-symmetry} by reducing to counts of proper curves satisfying a \emph{tail condition}.
We fix the data $\beta = (A ,\bu )$ consisting of a curve class $A\in \NE (X,\bbZ )$ and contact orders $\bu = (\bu_1,\bu_2 ,\bu_i)$ with $\bu_i=0$.

\begin{definition}
    An \emph{infinitesimal cylinder spine} $S$ of type $\beta = (A, \bu )$ is a $\Sigma$-piecewise linear map $h\colon \Gamma\rightarrow \vert \Sigma^{\ess}\vert$ such that:
    \begin{enumerate}[wide]
        \item the domain $\Gamma$ is a metric graph with a single $2$-valent vertex $v_0$, a single $3$-valent vertex $v_i$, an infinite contracted leg adjacent to $v_i$ and two finite legs with contact orders $\bu_1$ and $\bu_2$,
        \item the intersection of $h (\Gamma )$ and the support of $\Wall_A$ is either empty or equal to $h(v_0)$, and $h(v_0)\neq 0$,
        \item the $1$-valent vertices of $\Gamma$ are not mapped to rays of $\Sigma$,
        \item the $3$-valent vertex $v_i$ is mapped to the interior of a maximal cone and $h$ is balanced at $v_i$.
    \end{enumerate}
    A typical infinitesimal cylinder spine is depicted in \cref{fig:infinitesimal-cylinder-spine}.
\end{definition}

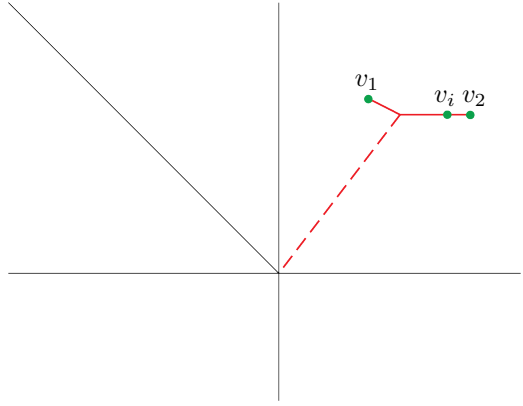
\begin{figure}
    \centering
    \setlength{\unitlength}{0.35\textwidth}
    \begin{picture} (1.5,1)
    %\includesvg[width=0.5\linewidth]{Infinitesimal-cylinder-spine.svg}
    \put(0,0){\includegraphics[height=\unitlength]{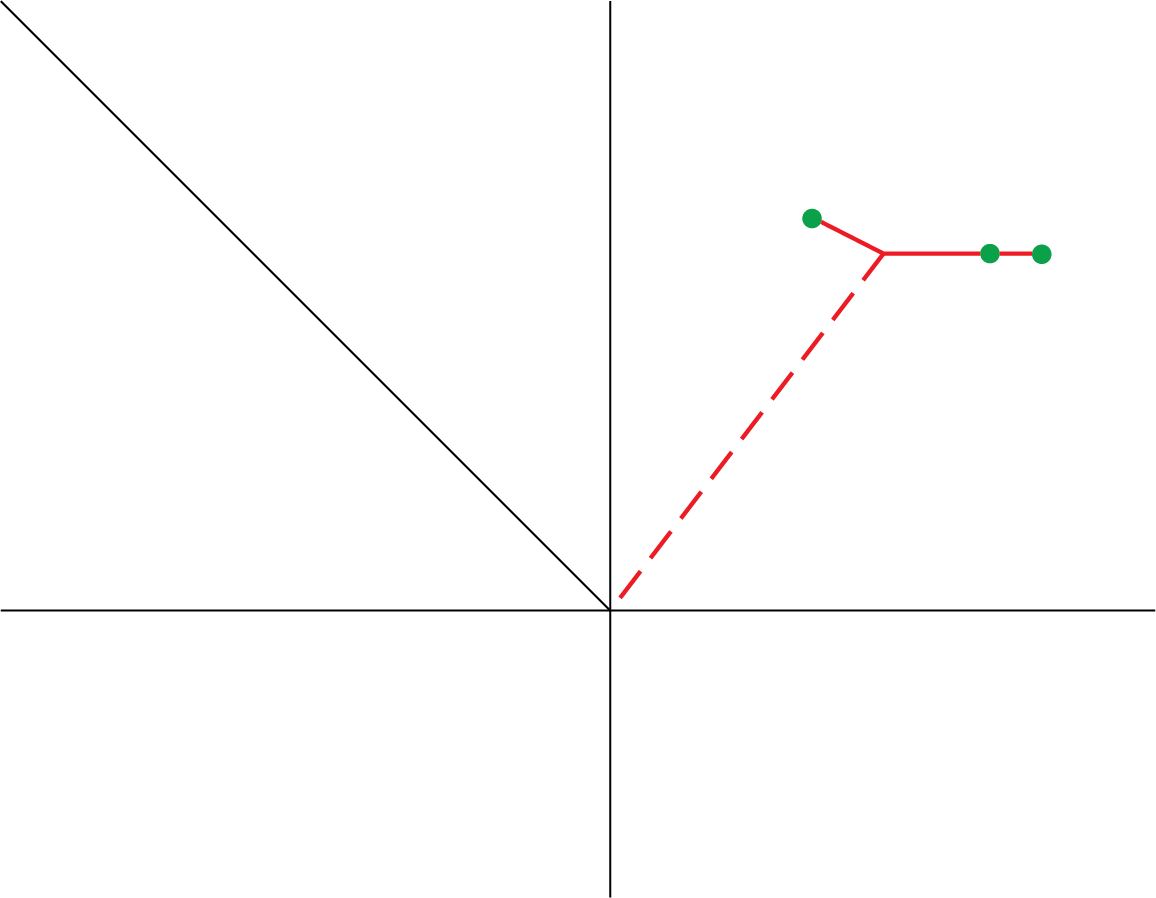}}
    \put(0.87,0.79){$v_1$}
    \put(1.07,0.75){$v_i$}
    \put(1.14,0.75){$v_2$}
    \end{picture}
    \caption{A typical infinitesimal cylinder spine (in solid red). To make it a tropical curve coming from a stable map, wall types are attached to the bending vertex (in dashed red).}
    \label{fig:infinitesimal-cylinder-spine}
\end{figure}

Given an infinitesimal cylinder spine $S$ of type $\beta = (A,\bu )$, the associated non-archimedean count is defined as 
\[N(S,A) \coloneqq \deg \big( \ev_i\colon \cM^{\sm} (U_{\eta}^{\an},\beta )'\cap \Sp^{-1} (S)\longrightarrow X_{\eta}^{\an} \big), \]
where $\cM^{\sm} (U_{\eta}^{\an} ,\beta )'$ is a moduli space of analytic stable maps to an auxiliary space, constructed from $X_{\eta}^{\an}$, which satisfy a generalized tail condition.
We refer to \cite[\S4 and \S8]{Keel-Yu_Log-CY-mirror-symmetry} for details.

While the construction of the count through an auxiliary space is well-suited for a definition in arbitrary dimension, the $2$-dimensional case allows for a direct definition as in \cite{Yu_Enumeration_of_holomorphic_cylinders_I}.
The affine structure on $\Sk (X,D)$ can be used to canonically extend an infinitesimal cylinder spine into a cylinder spine as in \cref{def:cylinder-spine}.
This produces a spine $\hS$, see \cite[\S4]{Yu_Enumeration_of_holomorphic_cylinders_I}, and the count $N(S,A)$ coincides with the non-archimedean cylinder count $N(\hS , \hA )$ where $\hA = A+\hdelta$ is a curve class determined by $\beta$.
We provide more details below. 

\begin{construction}[Extension of a spine] \label{constr:extension-spine}
    Let $S = (h\colon \Gamma\rightarrow \vert \Sigma\vert )$ be an infinitesimal cylinder spine.
    We construct a cylinder spine $\hS = (\hh\colon \hGamma\rightarrow\vert \Sigma\vert )$ called the \emph{extension of $S$} as follows.
    
    We glue infinite legs $[0,+\infty)$ to the finite $1$-valent vertices of $\Gamma$, require $\hh$ to have slope specified by the contact orders $\bu$ on those extended legs, and introduce the necessary $2$-valent vertices in order to make $\hh$ a $\Sigma$-piecewise linear map. 
    The reason why this procedure works is because the extended legs never hit the origin $0$, which is the singular locus of the integral affine structure, ensuring that we can still make sense of the contact orders $\bu$ in different cones of $\Sigma$.
    In the language of \cite{Yu_Enumeration_of_holomorphic_cylinders_I}, the spine $S$ is extendable.
    
    We associate to $\hS$ a curve class $\hA$ of the form 
    \[\hA = A + \hdelta_1 + \hdelta_2 ,\]
    where $\hdelta_1$ and $\hdelta_2$ are effective curve classes determined by the extended legs as in \cite[Definition 7.9]{Keel_Yu_The_Frobenius}.
    Then, the spine $\hS$ is a cylinder spine of type $\hbeta = (\hA ,\bu )$.
\end{construction}

The counts $N(S,A)$ and $N(\hS,\hA )$ are defined using different moduli spaces, but those moduli spaces are isomorphic.
While we do not prove this fact here, a consequence is that the counts $N(S,A)$ and $N(\hS ,\hA )$ agree. 

\begin{lemma}
    Let $(X,D)$ be a $2$-dimensional log Calabi-Yau pair satisfying the assumptions of \cref{subsec:log-setup}.
    Let $S$ be an infinitesimal cylinder spine of type $\beta = (A,\bu )$, and let $\hS$ be the cylinder spine of type $\hbeta = (\hA ,\bu )$ constructed from $S$.
    Then 
    \[N(S,A)  = N(\hS ,\hA ).\]
\end{lemma}

We remark that this alternative definition of $N(S,A)$ is only valid in the surface case.
In higher dimensions, singularities of the integral affine structure on $\Sk( X,D)$ are higher dimensional and prevent the naive extension procedure outlined above. 

Finally, we recall the following invariance property of infinitesimal cylinder counts.
The pairing $\langle \cdot , \cdot\rangle$ in the statement of the proposition is the canonical pairing between integral tangent vectors and integral affine functions.

\begin{proposition}[{\cite[Proposition 12.8]{Keel-Yu_Log-CY-mirror-symmetry}}]
    Let $S= (h\colon \Gamma\rightarrow\vert \Sigma^{\ess}\vert )$ be an infinitesimal cylinder spine of type $\beta = (A, \bu )$.
    Let $n$ be a primitive normal vector to the line through $0\in \Sk (U) \simeq\vert \Sigma^{\ess}\vert$ containing $h(v_0)$.
    Then the count $N(S,A)$ depends only on $A$, $\langle n,\bu_1 \rangle$ and $\bu_1 + \bu_2$.
\end{proposition}

\subsection{Scattering diagram of log Calabi-Yau surfaces}

In this subsection, we use the main comparison (\cref{thm:main-comparison}) and the splitting formula for broken line invariants in \cite[Eq.(4.2)]{Gross_Canonical-wall-structure-and-intrinsic-mirror-symmetry} to express infinitesimal cylinder counts in terms of wall types invariants in the surface case. 
As a consequence, we obtain the equivalence of the non-archimedean and logarithmic scattering diagrams (\cref{thm:exponential-formula}) and mirror constructions (\cref{coro:mirror-algebras}).

We first need to compute the invariants associated to a broken line type whose spine does not bend.
We recall that given a broken line type $\bomega$, the nonvanishing of the invariant $N_{\bomega}$ constrains the curve class decoration on vertices of the spine (\cite[Lemma 4.11]{Gross_Canonical-wall-structure-and-intrinsic-mirror-symmetry}).
The type $\bomega$ is called \emph{admissible} if the curve class decoration satisfies those contraints. 

\begin{lemma} \label{lemma:computation-straight-broken-line}
    Let $\bomega$ be an admissible decorated broken line type, let $u_{\out} = \bu(L_{\out} )$.
    If $\Sp (\bomega )$ is balanced, then $\bomega =\Sp (\bomega )$ and $k_{\omega}N_{\bomega} = 1$.
\end{lemma}

\begin{proof}
    We first prove that there are no wall types attached to vertices of the spine. 
    In the $2$-dimensional case, wall types are $0$-dimensional tropical types.
    It follows that if $\gamma$ is a wall type with leg $L_{\out}$, then $h_{\gamma} (\gamma_{\out} )$ is a ray starting at the origin $0$ and with slope vector $\bu (L_{\out} )$.
    Now, fix a vertex $v$ in the spine of $\bomega$.
    Denote by $\tau_1,\dots, \tau_k$ the wall types attached to $v$ and by $u_1,\dots, u_k$ their contact orders.
    We can write $u_i = m_i v_0$ for a primitive vector $v_0$ and positive integers $m_i >0$.
    Since $v$ is balanced the balancing condition implies that $\sum_i u_i = 0$ in the tangent lattice to $v$.
    It follows that all the contact orders $u_i$ must vanish, which contradicts the definition of a wall type. 
    We conclude that there are no wall types attached to $v$, hence that $\bomega$ coincindes with its spine. 

    To prove that $k_{\omega} N_{\bomega} =1$, we apply the splitting formula from \cite[Eq.(4.2)]{Gross_Canonical-wall-structure-and-intrinsic-mirror-symmetry} inductively to reduce to a trivial broken line type. 
    Let $v_0,\dots, v_r$ denote the vertices in $\bomega$, with $v_0$ adjacent to the punctured leg and $v_i$ connected to $v_{i+1}$ by an edge for all $0\leq i\leq r-1$.
    For $i\geq 1$, splitting at the edge connecting $v_{i-1}$ to $v_{i}$ produces two tropical types, and the type containing the vertex $v_{i}$ is a decorated broken line type which we denote by $\bomega_{i}$.
    Splitting at the vertex $v_r$ produces a trivial broken line  type (see \cite[Definition 3.19]{Gross_Canonical-wall-structure-and-intrinsic-mirror-symmetry}) which we denote by $\bomega_r$.
    Since there are no wall types, the splitting formula for broken lines gives:
    \[k_{\omega}N_{\bomega} = k_{\omega_1}N_{\bomega_1} = \cdots = k_{\omega_r}N_{\bomega_r}  =1,\]
    the last equality because $\bomega_r$ is a trivial broken line type.
\end{proof}

Since in the surface case infinitesimal cylinder counts can be expressed as extended cylinder counts, we obtain the following formula relating non-archimedean and punctured log invariants. 

\begin{proposition} \label{prop:infinitesimal-cylinder-counts-to-log-counts}
    Let $U$ be a connected smooth affine log Calabi-Yau surface with a compactification $(X,D)$ as in \cref{assumption:compactification}.
    Let $S$ be an infinitesimal cylinder spine in $\Sk (U)$ of type $\beta =(A,\bu )$ with $A\in\NE (X,\bbZ )$ and $\bu = (\bu_1 ,\bu_2, 0)$.
    Then 
    \[N(S,A) =  \sum_{\ell\geq 0} \sum_{\mu_1,\dots ,\mu_{\ell}\geq 0}\sum_{\substack{\btau_1,\dots ,\btau_{\ell}  }} \frac{\prod_{i=1}^{\ell}  \big( k_{\tau_i} W_{\btau_i}\big)^{\mu_i}}{\mu_1!\cdots \mu_{\ell}!}, \]
    where the last sum is over decorated wall types $\btau_1,\dots ,\btau_{\ell}$ with contact order $u_{\tau_i}$ and total curve class $A_i$ that satisfy: 
    \begin{align}
        \mu_1 u_{\tau_1} + \cdots +\mu_{\ell} u_{\tau_{\ell}}  &= \bu_1 +\bu_2 , \label{eq:contact-oders-exp-formula}\\
        \mu_1 A_1 + \cdots + \mu_{\ell} A_{\ell} &= A. \label{eq:curve-class-exponential-formula}
    \end{align}
\end{proposition}

\begin{proof}
    Let $\hS$ be the extended cylinder spine of type $\hbeta = (\hA , \bu)$ associated to $S$ (see \cref{constr:extension-spine}).
    By \cref{thm:main-comparison}, the non-archimedean count $N(S,A) = N(\hS ,\hA )$ decomposes as a sum of logarithmic cylinder counts
    \[N(S,A) = \sum_{\btau} k_{\tau} N_{\btau} ,\]
    where $\btau$ has total curve class $\hA$, contact orders specified by $\beta$, its $3$-valent vertex mapped to a maximal cone specified by $S$, and has spine type given by $\hS$.
    
    Let $\btau$ be one of the types appearing in the sum, and assume that $N_{\btau}\neq 0$.
    By \cref{lemma:splitting-cylinder-type}, splitting $\btau$ at the $3$-valent vertex produces two decorated broken line types $\bomega_1$ and $\bomega_2$ such that 
    \[ k_{\tau} N_{\btau}  =k_{\omega_1}k_{\omega_2}N_{\bomega_1}N_{\bomega_2} .\]
    Since $N_{\btau}\neq 0$, the counts $N_{\bomega_1}$ and $N_{\bomega_2}$ are nonvanishing.
    In particular, both $\bomega_1$ and $\bomega_2$ are admissible.
    By construction, one of the broken line types has a spine that contains only balanced $2$-valent vertices mapped to rays of $\Sigma^{\ess}$.
    We choose this type to be $\bomega_2$.
    By \cref{lemma:computation-straight-broken-line} we have $k_{\omega_2}N_{\bomega_2}=1$.
    For the type $\bomega_1$, applying the propagation rule \cite[Eq.(4.2)]{Gross_Canonical-wall-structure-and-intrinsic-mirror-symmetry} we find 
    \[k_{\omega_1}N_{\bomega_1} = k_{\omega_1'}N_{\bomega_1 '}\frac{\prod_{i=1}^{\ell}  \big( k_{\tau_i} W_{\btau_i}\big)^{\mu_i}}{\mu_1!\cdots \mu_{\ell}!} = \frac{\prod_{i=1}^{\ell}  \big( k_{\tau_i} W_{\btau_i}\big)^{\mu_i}}{\mu_1!\cdots \mu_{\ell}!},\]
    where $\bomega_1'$ is the broken line type obtained by splitting $\bomega_1$ at the bending vertex, and where $\btau_1,\dots, \btau_{\ell}$ are the decorated wall types attached to the bending vertex. 
    The second equality follows from \cref{lemma:computation-straight-broken-line}, because $\bomega_1'$ is admissible and every vertex in the spine of $\bomega_1'$ is balanced. 
    Furthermore, the balancing condition at the bending vertex guarantees that \eqref{eq:contact-oders-exp-formula} is satisfied, and \eqref{eq:curve-class-exponential-formula} is also satisfied because the total curve classes of $\bomega_1'$ and $\bomega_2$ correspond to the extension classes $\hdelta_1$ and $\hdelta_2$ in \cref{constr:extension-spine}.
    This can be seen by comparing \cite[Definition 4.10]{Gross_Canonical-wall-structure-and-intrinsic-mirror-symmetry} and \cite[\S 4]{Yu_Enumeration_of_holomorphic_cylinders_I}.
    The formula is proved.
\end{proof}

There are two wall-crossing functions attached to $x\in \Sk (U)\setminus 0$, one defined using wall types invariants and the other using infinitesimal cylinder counts. 
We first describe the ring in which those functions exist.
Let $M$ denote the integral tangent space to $x$, which is isomorphic to $\bbZ^2$, and fix a strictly convex toric monoid $Q\subset N_1 (X)$  containing $\NE (X,\bbZ )$.
We consider the completion of the ring $\bbZ [Q\oplus M]$ with respect to the maximal monoid ideal $I = \big( t^A z^u , A\in Q\setminus 0 , u\in M\big)$, that is:
\[\hR \coloneqq \lim_k \frac{\bbZ [Q\oplus M]}{I^k}.\]

The \emph{logarithmic wall-crossing function at $x$} is obtained as the exponential of a generating series of wall types invariants
\begin{equation} \label{eq:wall-crossing-function-log}
    f_x^{\log} (t,z)= \exp\bigg( \sum_{\btau} k_{\tau} W_{\btau} t^A z^{-u_{\tau}} \bigg) \in \hR,
\end{equation}
where the sum is over decorated wall types $\btau$ with total curve class $A$ and contact order $u_{\tau}$ at the unique leg $L_{\out}$, which satisfy $x\in h_{\tau} (\tau_{\out} )$.

The \emph{non-archimedean wall-crossing function at $x$} is obtained as a generating series of infinitesimal cylinder counts.
We denote by $n$ an integral normal vector to the ray connecting the origin $0$ to $x$, and define
\begin{equation} \label{eq:wall-crossing-function-NA}
    f_x^{\an} (t,z) = \sum_{\substack{\bw\in n^{\perp}\\ A\in \NE (X,\bbZ)}} N(S,A) t^A z^{\bw} \in \hR,
\end{equation}
where $S$ is any infinitesimal cylinder spine with contact orders $(\bu_1,\bu_2 , 0)$ such that $\bu_1+\bu_2=\bw$ and $\langle n ,\bu_1\rangle = 1$, and with the bending vertex mapped to $x$.

The main theorem of this section asserts that both wall-crossing functions agree. 
We call it the exponential formula, because it expresses the non-archimedean wall-crossing function as the exponential of a generating series of punctured log invariants.

\begin{theorem}[Exponential formula] \label{thm:exponential-formula}
    Let $U$ be a connected smooth affine log Calabi-Yau surface.
    For any $x\in \Sk (U)\setminus 0$, the non-archimedean and logarithmic wall-crossing functions at $x$ agree: 
    \[f_x^{\an} (t,z) = f_x^{\log} (t,z) .\]
    Consequently, the non-archimedean scattering diagram from \cite{Keel-Yu_Log-CY-mirror-symmetry} and the logarithmic scattering diagram from \cite{Gross_Canonical-wall-structure-and-intrinsic-mirror-symmetry} coincide.
    In particular, for any curve class $A\in \NE( X,\bbZ )$, we have $\Wall_A^{\an} = \Wall_A^{\log}$.
\end{theorem}

\begin{proof}
    The theorem follows from the formula in \cref{prop:infinitesimal-cylinder-counts-to-log-counts} and the definition of the wall-crossing functions.     
\end{proof}

The theta functions of a scattering diagram are defined using broken lines and wall-crossing functions (see \cite[\S3]{Gross_Canonical_bases}).
It is a key result of the logarithmic and non-archimedean mirror constructions that the scattering diagrams are theta function consistent.
The mirror algebra can then be recovered from the scattering diagram by taking the algebra generated by theta functions.
Thus, as a direct corollary of the exponential formula, we obtain the equivalence of the non-archimedean and logarithmic mirror constructions.

\begin{corollary} \label{coro:mirror-algebras}
    Let $U$ be a connected smooth affine log Calabi-Yau surface.
    Then the non-archimedean and logarithmic mirror constructions provide the same mirror algebra.
\end{corollary}

\begin{proof}
    The corollary follows from the fact that the mirror algebras are uniquely determined by the scattering diagrams through theta functions.
\end{proof}

\bibliographystyle{plain}
\bibliography{dahema}
\end{document}